\documentclass[11pt]{amsart}

\usepackage{amsbsy, amsmath,amsfonts,amssymb,amsthm,amscd,amssymb,latexsym}
\usepackage[latin1]{inputenc}
\usepackage{graphicx}
\usepackage[dvips]{epsfig}
\usepackage{graphics}
\usepackage{psfrag} 

\usepackage{cancel}

\usepackage{cancel}

\usepackage{enumerate}
\usepackage{indentfirst}
\usepackage{euscript}

 \theoremstyle{plain}
\newtheorem{proposition}{Proposition}
\newtheorem{lemma}{Lemma}

\newtheorem{remark}{Remark}
\newtheorem{definition}{Definition}
\newtheorem{theorem}{Theorem}

\newcommand{\m}{\displaystyle}

\newcommand{\R}{\mathbb{R}}
\newcommand{\M}{\mathbb{M}}

\def\la#1#2{\lambda_{#1#2}}

\def\gf#1#2{g_{#1#2}}

\def\uu{(u_1,u_2,u_3)}

\def\pa#1{\dfrac{\partial \alpha}{\partial u_{#1}}}

\def\pad#1#2{\dfrac{\partial ^2
\alpha}{\partial u_{#1}\partial u_{#2}}}

\def\ui#1{u_{#1}}
\def\func#1#2#3#4{{#1}^{#2}\rightarrow{#3}^{#4}}
\def\we{\wedge}
\def\fu#1#2{{#1}_{#2}}

\def\lp#1{${\mathcal L}_{#1}(\alpha)$}
\def\fp#1{${\mathcal F}_{#1}(\alpha)$}
\def\fuu#1#2#3{#1_{#2#3}}

\def\prf#1#2{\dfrac{\partial #1}{\partial #2}}

\def\re{{\mathbb R}}

\usepackage[usenames,dvipsnames]{xcolor}

 \title[Partially umbilic singularities   of
  hypersurfaces of  $\mathbb R^4$  ]{
Partially umbilic singularities   of
  hypersurfaces of  $\mathbb R^4$}

 \author{D. Lopes, J. Sotomayor and R. Garcia }

\thanks{The first author was partially supported by a doctoral fellowship  CAPES/CNPq.
The second author participated in the    FAPESP Thematic Project
2008/02841-4 and has a  fellowship  CAPES PVNS at UNIFEI.  
The authors are  partially supported by Pronex/FAPEG/CNPq  Proc.   2012 10 26 7000 803. 
 The  second and third authors are fellows of CNPq.
}

\begin{document}

\begin{abstract}
This paper establishes the
geometric
structure of the lines of principal curvature of
a
hypersurface
 immersed in   ${\mathbb R}^4$
in a neighborhood of the set   $\mathcal{S}$
of its  {\em principal curvature singularities}, 
  consisting  of   the points at which atF
least two principal curvatures are equal.
 Under generic
conditions defined by appropriate transversality  hypotheses it is
proved that $\mathcal{S}$ is the union of  regular smooth curves
$\mathcal{S}_{12}$ and $\mathcal{S}_{23}$,
 consisting of  {\em partially umbilic points},
 where  only
two principal curvatures coincide.
  This curve is   partitioned into regular arcs consisting of
  points of Darbouxian  types
 $D_1,\; D_2,\; D_3$,  with common boundary
 at isolated   semi-Darbouxian 
 transition 
 points of   types
  $ D_{12}$ and $D_{23}$.
   The
   stratified
structure of the {\em partially umbilic separatrix surfaces}, 
consisting of  the boundary of the set of  points through which
the principal lines approach  $\mathcal S$, established in this
work,  extends to hypersurfaces in ${\mathbb R}^4$
the results  of Darboux in  \cite{da} for  umbilic points  on
analytic surfaces in ${\mathbb R}^3$,  reformulated  by  Gutierrez
and Sotomayor in \cite{gs1}, to describe the umbilic separatrix structures of the 
umbilic types $D_1,\; D_2,\; D_3$,  and further developed by Garcia, Gutierrez and Sotomayor in \cite{gs4},
   for   their $ D_{12}$ and  $D_{23}$  generic bifurcations.
This work complements results of Garcia \cite{ga3} on the structure  of principal curvature lines around the generic partially umbilic points of hypersurfaces  in ${\mathbb R}^4$. 
 
 \end{abstract}

      \maketitle
 \section{ INTRODUCTION}\label{sec:int}

Let $ M^3$  be a
  $C^\infty$,  oriented, compact,
$3-$~dimensional manifold.

An immersion $\alpha$ of $ M^3$ into ${\mathbb R}^{4}$ is a map such that
$D{\alpha}_p : TM^3_p  \to {\mathbb R}^{4} $ is one to one, for every $p\in
M^3$.
 Denote by $Imm^k(M^3,{\mathbb R}^{4})$ the set
of $C^{k}$ - immersions of $ M^3 $ into ${\mathbb R}^{4}$  endowed with
the $C^k-$
 topology, see  \cite{levine}.

Associated to every $\alpha \in
Imm^k(M^3,{\mathbb R}^{4})$ is defined the normal map ${\mathcal N}_{\alpha}: M^3 \to
S^3$ :
 \vspace{-0.2cm}
$$ {\mathcal N}_{\alpha} =   \left( {\alpha}_1\we  {\alpha}_2\we {\alpha}_3  \right)/
 \mid
{\alpha}_1\we {\alpha}_2 \we{\alpha}_3\mid,$$
\noindent  where $({u_1},u_2,{u_3}): (M, p ) \to ({\mathbb R}^3, 0 ) $
 is a positive chart of $ M^3$  around $p$, $\we$ denotes the
product of vectors
  determined by a once for all fixed orientation  in ${\mathbb R}^{4}$.  This space
 is endowed  with the Euclidean norm $ | \cdot | =  \langle   \cdot ,   \cdot \rangle ^{\frac 12}$.
Also
 ${\alpha}_1 = \prf{\alpha}{u_1}$, ${\alpha}_2 = \prf{\alpha}{u_2}$, $ {\alpha}_3 =
\prf{\alpha}{u_3} $.

Clearly, ${\mathcal N}_{\alpha}$ is well defined  and of  class $C^{k-1}$ in
$M^3$.

Since $D{ \mathcal N}_{\alpha}(p)$ has its image contained in that of $D\alpha
(p),$ the endomorphism
${\omega}_{\alpha} : TM^3 \to TM^3 $ is well 
defined by
\vspace{-0.2cm}
$$ D\alpha.{\omega}_{\alpha} = D{\mathcal N}_{\alpha}.\vspace{-0.2cm}$$

It is well known that ${\omega}_{\alpha}$ is  a self  adjoint endomorphism,
when $ TM^3 $ is endowed with the metric ${\langle \cdot, \cdot \rangle}_{\alpha} $ induced by
$\alpha$ from the metric in ${\mathbb R}^{4}$. See \cite{sp2}.

The opposite values of the eigenvalues of ${\omega}_{\alpha}$ are
called  {\it principal curvatures} of  $\alpha$ and will be denoted by $k_1 = k_1 (\alpha)
\le k_2  = k_2  (\alpha)  \le k_3 =  k_3  (\alpha)$.

 The  {\em principal singularities}
  of  the immersion $\alpha$
 are defined as follows:

\begin{itemize}
\item[{$\bullet$}] \;{\em Umbilic Points}:  ${\mathcal U}_\alpha=\{ p\in M^3: k_1(p) = k_2(p) = k_3(p)\}$,  

\item[{$\bullet$}] \; {\em Partially Umbilic Points}:  ${\mathcal S}_\alpha={\mathcal S}_{12}(\alpha)\cup {\mathcal S}_{23}(\alpha)$,
where  
\item[{$\bullet$}] \; ${\mathcal S}_{12}(\alpha)=\{ p\in M^3: k_1(p) = k_2(p)
<
k_3(p)\}$,
\item[{$\bullet$}] \; $  {\mathcal S}_{23}(\alpha)=\{ p\in M^3: k_1(p)
< k_2(p) = k_3(p)\}.$
\end{itemize}

 The eigenspaces associated to the principal
curvatures,
when simple,
 define three
 line fields \lp i, $(i=1, 2, 3)$,
 mutually orthogonal in $TM^3$  (endowed with the metric $ {\langle \cdot, \cdot
\rangle}_{\alpha}$), called {\it principal line fields} of $\alpha$.
 They are
characterized by Rodrigues' equations  \cite{sp1, sp2} and  \cite{st}.
$$ 
\mathcal L_i(\alpha)  =
\{ v\in TM^3 : {\omega}_{\alpha}v +
k_i v = 0, \;  i= 1, \;  2,\;3 \}.$$
These line fields are well defined and smooth  outside
 their
respective
  sets
   of {\em principal singularities}, as follows:
 $\mathcal L_1(\alpha) $ is  of class
$\;\;C^{k-2}$ outside ${\mathcal U}_\alpha  \cup {\mathcal
S}_{12}(\alpha)$,

$\mathcal L_3(\alpha) $  is of class $\;\;C^{k-2}$ outside
${\mathcal U}_\alpha  \cup {\mathcal S}_{23}(\alpha)$,

$\mathcal L_2(\alpha) $ is of class $\;\;C^{k-2}$ outside
${\mathcal U}_\alpha  \cup {\mathcal S}_{\alpha}$.

This follows from the smooth dependence of simple eigenvalues and corresponding one-dimensional eigenspaces.

The integral curves of $\mathcal{L}_i,\; (i=1,2,3)$ are called the {\em principal foliations } $\mathcal{F}_i(\alpha)$ of $\alpha$.

Generically, for an open and dense set,  in the space
$Imm^{k}(M^3,{\mathbb R}^{4})$, ${\mathcal U}_\alpha = \emptyset$ and
${\mathcal S}_\alpha$, when non empty,  is  a regular submanifold of
codimension two, consisting of two pieces ${\mathcal S}_{12}
(\alpha)$ and ${\mathcal S}_{23} (\alpha)$.
This follows from the
Transversality Theorem \cite{thom} together with the following
well known facts: In the $6-$ dimensional space of $3\times 3$
symmetric matrices, those with three  equal eigenvalues, denoted
$\Sigma ^3$,  has codimension $5$ and those with only two equal
eigenvalues, denoted $\Sigma ^2$,  has codimension $2$. See 
Lax \cite{lax}. 
In  Section \ref{sec:summa} the symmetric matrices $\Sigma ^3$ and  
$\Sigma ^2$  will  appear
identified with the umbilic and partially umbilic $2-$jets,  $\mathcal U ^2$  and   $(\mathcal { PU} )^2$ and their extensions to the 
space of $4-$jets, where the submanifolds of interest for this work naturally belong.

 In this work  will be assumed  that the $2-jet$  extensions,
$j^2_{\alpha}$,  of the immersions $\alpha$
have
their $\omega _\alpha$ endomorphism transversal to both manifolds,
$\Sigma ^3$ and $\Sigma ^2$ described  above.
Coordinate
expressions for the  five types of generic partially umbilic
points (of codimension $\leq 3$)  and pertinent transversality
conditions in terms of the $4-jets$ of $\alpha$
 will be given in Definitions
 \ref{defDi}, \ref{defD12} and  \ref{defD23}. See also section \ref{sec:summa}.

Thus in this paper,  when non empty, ${\mathcal S}_\alpha $  will
be a
finite
 collection of
  closed
 regular
curves of class $C^{k-2}$ called the {\it partially umbilic
curves} of $\alpha$, the umbilic set $\mathcal U_\alpha$  being empty.

At each point $p\in {\mathcal S}_{12} $   is defined the 
\emph{
partially umbilic plane} ${\mathcal P}_3 (p) $  which is
orthogonal to the unique well defined principal direction
$\mathcal L_3 (p)$,  eigenspace  associated to the simple
eigenvalue $-k_3 (p)$.  Analogous  definition
holds
for $p\in
{\mathcal S}_{23} $, $-k_1 (p)$  and ${\mathcal P}_1(p)$.

Fig.  \ref{pucurve} illustrates a   partially umbilic curve $
{\mathcal S}_{12} $
 and  some
 of its
 possible contact behaviors with
the \emph{distribution} or \emph{field} of  planes ${\mathcal P}_3$  in  $M^3 \setminus {\mathcal S}_{23}$.
Similar
illustration applies to 
 ${\mathcal S}_{23} $  and the \emph{distribution}  or \emph{field}   of planes ${\mathcal P}_1$.

 This is an  initial schematic  illustration.
  Specific configurations of principal curvature lines,
  which are the
  integral foliations ${\mathcal
F_i }({\alpha})$ of the line fields  ${\mathcal L}_i ({\alpha})$,
 are established  in  this work for the generic case   around partially umbilic points.

 \begin{figure}[h]
 \begin{center}
     \def\svgwidth{0.75\textwidth}
        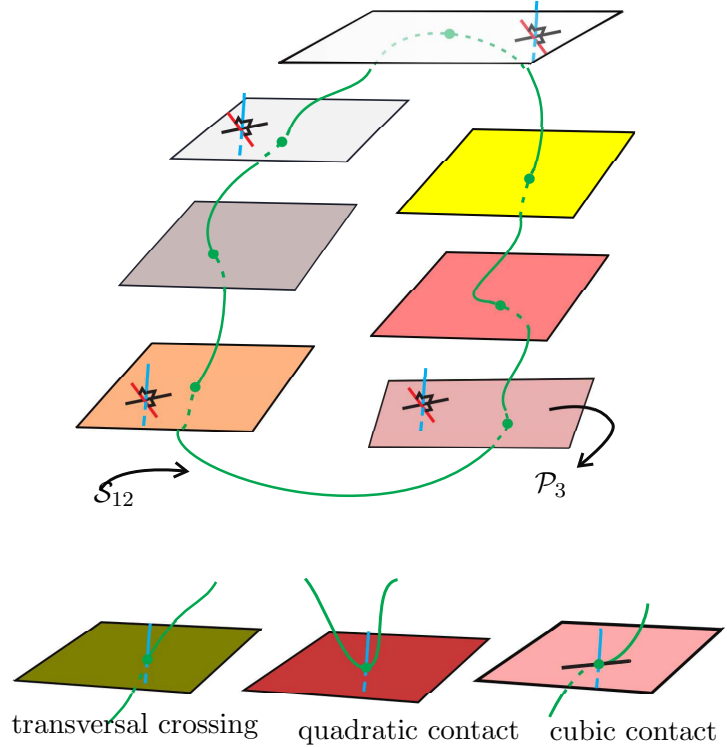
        
        \vspace{0.2cm}
  \caption {Partially umbilic curve, say ${\mathcal S}_{12}$, in green,
   and its partially umbilic planes ${\mathcal P}_3$, in multiple colors.
   Illustration of  its relative position with respect to the distribution of ${\mathcal
   P}_3$ planes:
   transversal crossings (almost everywhere),  quadratic
   contacts (top and bottom)
   and one  more degenerate, non-generic cubic contact,  middle right.}
    \label{pucurve}
     \end{center}
 \end{figure}

To make concrete this schematic picture 
the following 
must be added:

\noindent 1. - Along the transversal arcs: sub-arcs of Darbouxian points $D_i,$ \, $ i =1,\,2$,  are separated by  $D_{1,2}$  transition points, as established in Theorem \ref{th:d12}.\\

\noindent 2.- At the quadratic tangential points the sub-arcs of  Darbouxian $D_i ,$\,  $ i =2,\,3$ points,  are separated by   $D_{23}$  transition points.  See  Theorem \ref{th:d23}.\\

In this  work  will be established the \emph{ stratified  structure}    of the integral foliations ${\mathcal
F_i }({\alpha})$
near  the generic singularities ${\mathcal S}_\alpha$. See Mather  \cite{ma}.
To this end   will be provided  a  new approach and
proofs,  based  on the  Lie - Cartan suspension, explained in section
\ref{sec:2}, which allows to improve the results given in Garcia
\cite{ga1} and \cite{ga3}.

 See Theorem \ref{th:d123}, in section \ref{sec:dar_PU}, Theorem
\ref{th:d12}, in section \ref{sec:d12_PU} and Theorem  \ref{th:d23}, in section \ref{sec:d23_PU},  for a synthesis of the main  results of this work.
The proofs will be given after some preliminaries
 on the differential equations of principal foliations and
 the basic properties of  their Lie-Cartan suspensions  presented in section \ref{sec:2}.

Sections \ref{sec:4}, \ref{D12} and
  \ref{sec:d23}, respectively,  contain  the proofs of Theorems \ref{th:d123}, \ref{th:d12}  and \ref{th:d23}.   See also  the  summarizing  Theorem  \ref{th:summat} in section \ref{sec:summa}. 
  
  Section \ref{sec:CR}  is devoted to a discussion of  the 
     relation  of the results of this paper with  previous  and forthcoming ones.  
     The Appendix \ref{sec:App}  contains
     the    coordinate expressions of the geometric functions essential for the calculations in the proofs of the main results of this paper. 

 This work can
 also 
  be regarded as an extension  to $\mathbb R^4$ the results about  umbilic
points and their generic bifurcations  for surfaces immersed in $\mathbb R^3$.
 See  \cite{gas}, \cite{gs1},  \cite{gs4}, \cite{gs6}.

 The structure of the 
 global principal configurations on the
  ellipsoids in $\mathbb R^4$  has been established   Garcia, Lopes and Sotomayor  \cite{desoga}.

\subsection {Definition of the Color Convention in Illustrations.} \label{ss:CC}
 
The color convention for partially umbilic  curves and the  integral leaves  of principal line fields,  packed into strata, 
  established in this paper  is defined
as follows.

\begin{itemize}
\item [] Black: integral curves of line field  $\mathcal L_1$,

 \item [] Red: integral curves of line field $\mathcal L_2$,

 \item []Blue:  integral curves of line field  $\mathcal L_3$,

 \item[] Green: Partially umbilic arcs $\mathcal S_{12}$,

 \item[] Light Blue: Partially umbilic arcs $\mathcal S_{23}$.

 \end{itemize}

\section{Adapted Monge charts  at   Partially Umbilic Points} \label{sec:Monge_jet}

  Let $ p\in M^3 $  be  a partially 
umbilic point of an immersion $\alpha$ such that $ k_1(p)= k_2(p)=k(p) <  k_3(p) $. That is $p \in \mathcal S_{12}(\alpha)$.

\noindent  Let $\uu : M^3 \rightarrow {\mathbb R}^3 $ be a local chart
and  $ R : {\mathbb R}^4 \to {\mathbb R}^4 $ be an  isometry such that :
 $$
  (R\circ \alpha)\uu = (u_1, u_2, u_3, h \uu)
 $$
 where:
\begin{equation}\label{eq:di}
\begin{array}{ll}
h &=\displaystyle\frac{k}{2}(u_1^2+u_2^2)
+\displaystyle\frac{k_3}{2}u_3^2+\displaystyle\frac{a}{6}u_1^3+
\displaystyle\frac{b}{2}u_1u_2^2+\displaystyle\frac{c}{6}u_2^3+\displaystyle\frac{q_{003}}{6}u_3^3+ \displaystyle\frac{q_{012}}{2}u_2u_3^2\\
&+q_{111}u_1u_2u_3
+\displaystyle\frac{q_{021}}{2}u_2^2u_3+\displaystyle\frac{q_{102}}{2}u_1u_3^2
+\displaystyle\frac{q_{201}}{2}u_1^2u_3+\displaystyle\frac{A}{24}u_1^4
+\displaystyle\frac{B}{6}u_1^3u_2\\
&+\displaystyle\frac{C}{4}u_1^2u_2^2+\displaystyle\frac{D}{6}u_1u_2^3
+\displaystyle\frac{E}{24}u_2^4+\displaystyle\frac{Q_{004}}{24}u_3^4+
\displaystyle\frac{Q_{013}}{6}u_2u_3^3+\displaystyle\frac{Q_{103}}{6}u_1u_3^3\\
&+
\displaystyle\frac{Q_{022}}{4}u_2^2u_3^2
 +\displaystyle\frac{Q_{202}}{4}u_1^2u_3^2+\displaystyle\frac{Q_{112}}{2}u_1u_2u_3^2+
\displaystyle\frac{Q_{031}}{6}u_2^3u_3+\displaystyle\frac{Q_{301}}{6}u_1^3u_3\\
&+\displaystyle\frac{Q_{121}}{2}u_1u_2^2u_3+\displaystyle\frac{Q_{211}}{2}u_1^2u_2u_3 + h.o.t.
\end{array}
\end{equation}

 \begin{remark} \label{rem:adapted}
The rotation $R$ was chosen to eliminate the coefficient of the term $u_1^2 u_2$.
In this sense the chart
 is
 said to be  {\em adapted by rotation}.

Composing  the immersion with a   homothety ,  one of the coefficients $k$ or $k_3$  in equation  (\ref{eq:di})
can be taken to be $1$ or $-1$. After a reflection,  $1$ can always be chosen to be such coefficient.
In this sense the chart is  said to be  {\em adapted by homothety  and reflection}.

By means of an additional inversion  the other coefficient  (i.e.  $k_3$ or $k$) may be assumed to be $0$.

 Some of the long expressions appearing in this work may be simplified  by using {\em adapted charts}.
\end{remark}

\section{Darbouxian Partially Umbilic Points} \label{sec:dar_PU}

\begin{definition}\label{defDi}
The point $p$
in ${\mathcal S}_{12}(\alpha)$
 is called  a  Darbouxian  partially umbilic point of
type   $D_i$ if, in the notation of equation (\ref{eq:di}),
 the
 geometric  transversality
 condition $T$
 and
  the discriminant condition
  $D_i$ below hold.

\begin{itemize}
\item[T)] $ b (b - a ) \ne 0 $
\end{itemize}

\begin{itemize}
\item[D$_1$)] $\displaystyle\frac{a}{b}>\left(\frac{c}{2b}\right)^2+2$;
\item[D$_2$)] $1<\displaystyle\frac{a}{b}<\left(\frac{c}{2b}\right)^2+2$,\; $a\neq 2b$;
\item[D$_3$)] $\frac{a}{b}<1$.
\end{itemize}
\end{definition}

The main result of this
section is stated now.

 \begin{theorem}   \label{th:d123} Let $\alpha \in {Imm}^k(M^3, {\mathbb R}^4 ) \quad $
and $p\in {\mathcal S}_{12}(\alpha).$
 Then there is a  neighborhood $V_p$ of $p$ where ${\mathcal S}={\mathcal S}_{12}(\alpha) \cap V_p$  is a smooth curve
consisting of points $D_i$
where it holds  that:

 \begin{itemize}
 \item[i)] For the case  $D_1, $
 there exists  a unique
invariant separatrix  surface $W_1({\mathcal S})\subset V_p$, $\partial W_1 ({\mathcal S})={\mathcal S}$,
of class  $C^{k-3},\;$ fibred over
$\mathcal S  $ whose fibers are  leaves  of  \fp 1.
Only these leaves (the fibers) are asymptotic to the partially umbilic curve ${\mathcal S}$.

 The set $V_p \backslash W_1({\mathcal S})$ is a  hyperbolic sector of  \fp 1. See Fig.
 \ref{configD1D2D3Sf1}, left.

 \item[ii)]  For the case  $D_2$
there exist  two invariant separatrix surfaces as described in  item i) and exactly one wedge sector and
one hyperbolic sector of \fp 1.  See Fig. \ref{configD1D2D3Sf1}, center.

 \item[iii)]
 For the case  $D_3$,  there exist
three invariant  separatrix surfaces  as in item i)  and exactly three  hyperbolic sectors of \fp 1.
 See Fig. \ref{configD1D2D3Sf1}, right.

\item[iv)] The same conclusions  hold  for the foliation  \fp 2 which is
orthogonal to \fp 1 and singular in the set ${\mathcal S_{12}}$.
 See Fig. \ref{configD1D2D3Sf2}.

\end{itemize}
\end{theorem}

\begin{figure}[h]
\psfrag{S}{${\mathcal S} = {\mathcal S _{12}}$}
\psfrag{W}{$W_1({\mathcal S})$}
\psfrag{W1}{$W_1({\mathcal S})$}
\psfrag{W2}{$W_2({\mathcal S})$}
\psfrag{W3}{$W_3({\mathcal S})$}
\begin{center}
 \def\svgwidth{0.85\textwidth}
    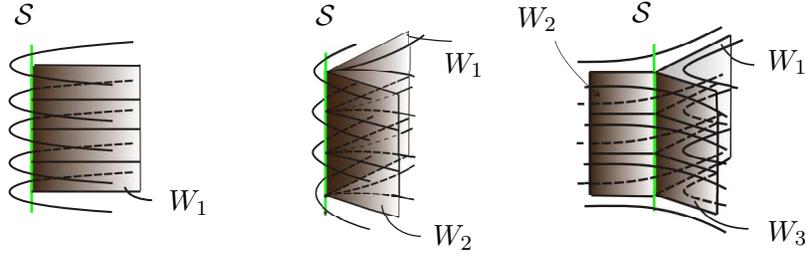
    \caption{Principal Foliations $\mathcal{F}_1(\alpha)$ and partially
    umbilic  separatrix surfaces in the
neighborhood of the point $D_1$ (left), $D_2$ (center) and $D_3$ (right).}
  \label{configD1D2D3Sf1}
    \end{center}
\end{figure}

\begin{figure}[h]
\psfrag{S}{${\mathcal S}$}
\psfrag{W}{$W_1$}
\psfrag{W1}{$W_1$}
\psfrag{W2}{$W_2$}
\psfrag{W3}{$W_3$}
\psfrag{W13}{  $W_2$}
\psfrag{W23}{ $W_3$}
\psfrag{W33}{$W_1$}
\begin{center}
 \def\svgwidth{0.85\textwidth}
  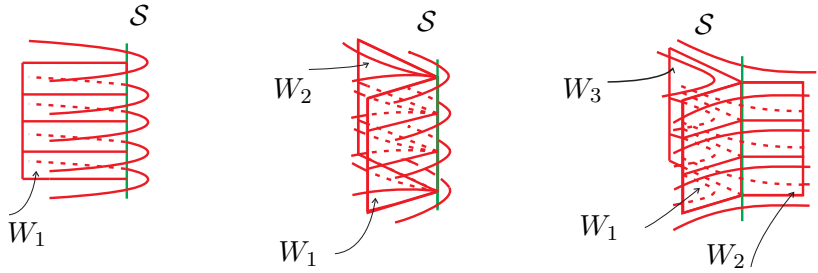
    \caption{Principal Foliations $\mathcal{F}_2(\alpha)$ and partially
    umbilic  separatrix surfaces  in the
neighborhood of the point $D_1$(left), $D_2$(center) and $D_3$(right).}
  \label{configD1D2D3Sf2}
    \end{center}
\end{figure}

\begin{definition}
A separatrix surface as that in theorem \ref{th:d123} is called partially umbilic  separatrix surface.
\end{definition}

\section{$D_{12}$ Partially Umbilic Points} \label{sec:d12_PU}

\begin{definition}\label{defD12} Let $ p $ be a partially umbilic point and
$\alpha$  expressed  as in equation \eqref{eq:di}.
The point $p $ is called    semi-Darbouxian of  type
 D$_{12} $ if the following conditions are satisfied:

\begin{itemize}
\item[ D$_{12})$] $ b c (b-a)\neq0$, $\m\frac{a}{b}=2$  and  $\chi_{12} \neq 0$,
where

\begin{equation}
\label{chi2} 
\aligned \chi_{12}= &\left(k-{   k_3} \right)  \left[   \left(b{  q_{021}} -b{  q_{201}}-c{
  q_{111}} \right) { B}-\,b {   q_{111}}\; C+  b^2 Q_{211}\,
 \right] \\
 +&    q_{012}\,   q_{201}\,{b}^{2}+ \left(2\, q_{102}\,{b}^{2}-b{k}^{3}{ k_3}+b{k}^{4} \right)  q_{111} -3\,c q^{2}_{111}  q_{201}\\
 -&3
 \,b q_{111}\, q^2_{201}-2b\,{  q^3_{111}}  +2 b\,  q_{111}\,
  q_{021}\,  q_{201} 
 \endaligned
\end{equation}

\end{itemize}
\end{definition}

\begin{remark}  When $ \frac{a}{b}=\left(\frac{c}{2b}\right)^2+2$, $b(b-a)\ne 0$, another
condition $\chi_{12}^{*}=\chi_{12}^{*}(j^4h(0))\neq 0 $ characterizes a $D_{12}$ partially umbilic point, see equation (\ref{chi11}) in the subsection \ref{chi1s}.  
This condition 
can be obtained   from  that given in defining equation (\ref{chi2} )    by an appropriate rotation  in the $(u,v)$-plane in equation (\ref{eq:di}).  In fact   the two patterns of  failure  of the  discriminant condition  $D_2$,  $ \frac{a}{b}=\left(\frac{c}{2b}\right)^2+2$  and $ a=2b$,  keeping the transversality condition $b\neq a$,   are permuted by a rotation.

The calculations of these conditions give shorter expressions working with $\chi_{12}$.

\end{remark}

\begin{remark} \label{rem:chi-longo12}
The condition $\chi_{12}\neq0$ expresses the transversality, at $D_{12}$,
 of the transition between types $D_1$ and
$D_{2}$,
as will follow from the analysis in section \ref{D12}.
As discussed in remark \ref{rem:adapted} the long expressions  in (\ref{chi2})  and (\ref{chi11}) can be simplified by taking $k_3 =1$ and $k = 0$ by applying to it an inversion and a homothety .
\end{remark}
 
\begin{theorem}\label{th:d12}    Let  $p$ be a $\fuu D12$  partially umbilic point.
Then it has  a
neighborhood  $V_p $ which intersects $\mathcal S_\alpha$
 on a   partially umbilic smooth curve
 ${\mathcal S}={\mathcal S}_1\cup {\mathcal S}_2\cup \{p\}$,  transversal to ${\mathcal P}_3(p)$,
 which consists on two arcs of   points 
 of types
 $ D_1$ and $ D_2$,
 having $p$ as common boundary.
 It 
 holds
 that:

\begin{itemize}
\item[i)] There exists a  partially umbilic separatrix surface of   \fp 2, $W = W_1({\mathcal S})\cup F_1 \cup W_2({\mathcal S})$,
 stratified as follows:

\noindent $\bullet$   \; $  W_1 ({\mathcal S})$  is a partially umbilic separatrix surface of an arc of  partially umbilic points, ${\mathcal S}_1$,  of type $D_1$.

\noindent $\bullet$   \; $ W_2 ({\mathcal S})$  is a partially umbilic separatrix surface of an arc  ${\mathcal S}_2$  of  partially umbilic points of type $D_2$.

\noindent $\bullet$   \; $F_1$  is a simple  leaf of \fp 2 asymptotic to the partially umbilic point $p$ of type $D_{12}$. 
 See Fig. \ref{D12S}.
 
\item[ii)]  There exists a  partially umbilic separatrix surface $ W_3({\mathcal S})\cup F_2  $ of \fp 2,
  stratified as follows:

\noindent $\bullet$   \; $ W_3 ({\mathcal S})$  is a partially umbilic separatrix surface of an arc ${\mathcal S}_2$  of  partially umbilic points of type $D_2$.

\noindent $\bullet$   \; $F_2$  is a double leaf of \fp 2 asymptotic to the partially umbilic point $p$ of type $D_{12}$.  See Fig. \ref{D12S}

  \item[iii)] There exists a  three dimensional wedge sector $\mathcal W$  such that
 $\partial {\mathcal W} $  is a variety 
 partitioned  
  into  strata of dimension two, one and zero,
 as follows:
 
 \noindent $\bullet$ \; $W_4({\mathcal S})\cup W_3({\mathcal S})\cup  W_2({\mathcal S}) $ 
 are the bi-dimensional strata. Moreover $W_4({\mathcal S})$ 
 consists 
 of leaves of \fp 2 which are asymptotic to the partially umbilic point $p$ of type $D_{12}$. See Fig.  \ref{fig:D12Estratificado}.

  \noindent $\bullet$ \; $F_1\cup F_2\cup {\mathcal S}_2$ are the one dimensional strata.

   \noindent $\bullet$ \; $p$ is the 
   zero 
    dimensional stratum.

\item[iv)]  The same conclusions  
hold 
for the foliation  \fp 1 which is
orthogonal to \fp 2 and singular in the set ${\mathcal S}$.   See Fig. \ref{D12S} (left).

 The behavior of the foliations  \fp 1 and \fp 2 a neighborhood $V_p $
of  $p$ is illustrated  in   Fig. \ref{D12S}. 
The stratification of the wedge  sector is illustrated in Fig. \ref{fig:D12Estratificado}.

\end{itemize}

\end{theorem}

\begin{figure}[h]
\centering
  \def\svgwidth{0.30\textwidth}
      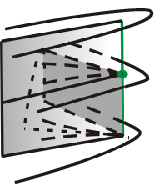
        \def\svgwidth{0.45\textwidth}
    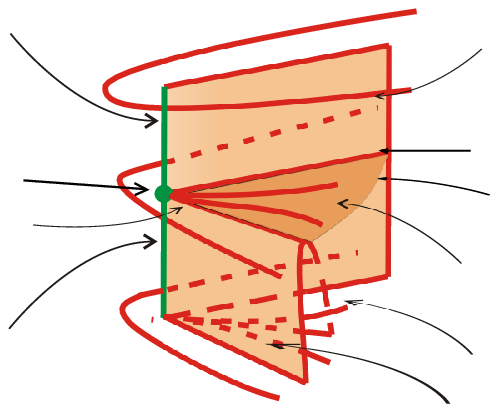
    \caption{Principal foliations $\mathcal{F}_1(\alpha)$ (Left) and $\mathcal{F}_2(\alpha)$(right)  in 
    a neighborhood of the point $D_{12}$.}
  \label{D12S}
\end{figure}

\begin{figure}[h]
\begin{center}
\hspace{-2cm}
    \def\svgwidth{0.80\textwidth}
    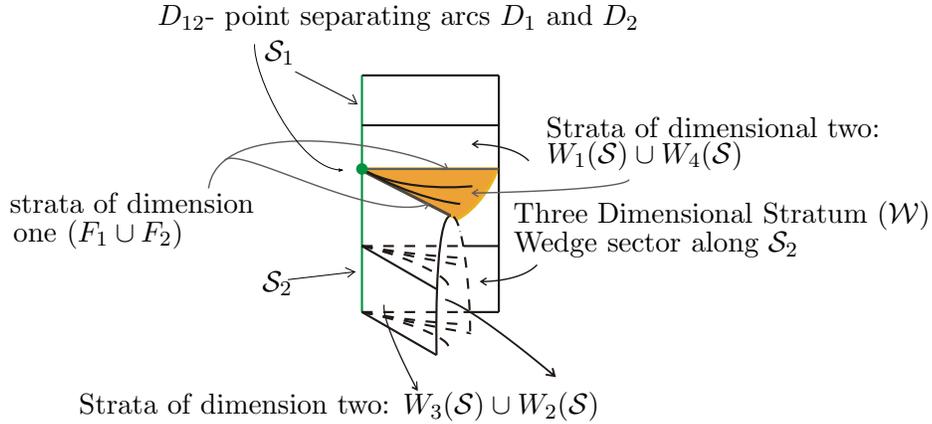
    \caption{Stratification of the wedge sector (strata of dimensions three, two and  one) and umbilic separatrices of  a 
    $D_2-$ $D_{12}-$ $D_1$ 
     partially umbilic   arc 
     of ${\mathcal F}_{2}(\alpha)$. }
  \label{fig:D12Estratificado}
    \end{center}
\end{figure}

\begin{definition} 
The leaf $F_1$ in the item $i)$ of theorem \ref{th:d12} is called
a simple  $D_{12}$ separatrix, and  the leaf
 $F_2$ in the item $ii)$   is called a double $D_{12}$
  separatrix. There are no leaves tangent to  $F_1$ and asymptotic to the point $D_{12}$. There are infinitely many leaves tangent to $F_2$ and asymptotic to the point $D_{12}$.

\end{definition}

\section{$D_{23}$ Partially Umbilic Points} \label{sec:d23_PU}

\begin{definition}\label{defD23}
Let $ p $ be a partially umbilic point and $\alpha$ parametrized as
 in equation \eqref{eq:di}. The point $p$ is called {semi-Darbouxian of  type}
 $ D_{23} $ if the following conditions hold:
\begin{itemize}
\item[ D$_{23})$] $b=a\neq 0$,  $ b(q_{201} - q_{021}) + c q_{111}\neq 0$ and $\chi_{23}\ne 0$ where,

\begin{equation}\label{chid23}
\aligned 
\chi_{23}  = &\left(k-{  k_3} \right)  \left({  b A } +c  B  -b{ C }-2 \,b{k}^{3} \right) \\
+& \left(3\,{  q^2_{201}}-{  q_{201}}\, q_{021}-
2\,{ q^2_{111}} \right) b+3\,q_{111}\, q_{201}\,c
\endaligned
\end{equation}

\end{itemize}
\end{definition}

\begin{theorem} \label{th:d23} Let  $p$ be a   
  $\fuu D23$ partially umbilic point.
Then there exists a  neighborhood $V_p$ of $p$ where $\mathcal S \cap V_p={\mathcal S}_1\cup {\mathcal S}_2 \cup \{p\} $  is a smooth curve tangent to the
partially
umbilic plane,
consisting of two arcs  of partially umbilic points of types  $D_2$ and  $D_3$,  separated by the $D_{23}$ point.
The following holds in $V_p$.

\begin{itemize}
\item[i)] There exists a regular surface $W_1(\mathcal S) \subset V_p $ containing
the  partially umbilic curve $\mathcal S$
with
the following property:
a connected component of $  W_1(\mathcal S)\setminus {\mathcal S}$
 is invariant by \fp 1
and the other is invariant by \fp 2.  See
Fig. \ref{D23S}.

\item[ii)] There exists a regular surface $W_2(\mathcal S) \subset V_p $ containing
the  partially umbilic curve $\mathcal S$
with the following property:
one of 
the
connected components of $  W_2(\mathcal S)\setminus {\mathcal S}$
 is invariant by \fp 1 and the other is invariant by \fp 2.  See
Fig. \ref{D23S}.

\item[iii)] There exists a surface $W_3(\mathcal S) \subset V_p $ containing
the  partially umbilic curve $\mathcal S$
with the following property:
one of 
the
connected components of $  W_3(\mathcal S)\setminus {\mathcal S}$
 is invariant by \fp 1 and the other is invariant by \fp 2. 

Moreover $F_5\subset W_3$ is a leaf of \fp 2 asymptotic to the $D_{23}$ partially umbilic point $p$.
  See
Fig. \ref{D23S}.

    \item[iv)] There exists a  three dimensional wedge sector ${\mathcal W}_1$  such that
 $\partial {\mathcal W}_1 $  is a variety (union of strata of dimension two, one and zero) partitioned  as follows:
 
 \noindent $\bullet$ \; $ W_2({\mathcal S})\cup  W_1({\mathcal S}) $ are the bi-dimensional strata. See Fig.   \ref{D23Black}.
 
  \noindent $\bullet$ \; $F_1\cup F_2\cup {\mathcal S}_1\cup {\mathcal S}_2$ are the one dimensional strata.
  
   \noindent $\bullet$ \; $p$ is the one dimensional stratum.

   \item[v)] There exists a  three dimensional wedge sector ${\mathcal W}_2$  such that
 $\partial {\mathcal W}_2 $  is a variety  partitioned  
 into strata of dimensions  two, one and zero, as follows:
 
 \noindent $\bullet$ \; $W_4({\mathcal S})\cup W_2({\mathcal S})\cup  W_1({\mathcal S}) $ are the bi-dimensional strata. Moreover $W_4({\mathcal S})$ 
 consists 
 of leaves of \fp 2 which are asymptotic to the partially umbilic point $p$ of type $D_{23}$.  See Fig.   \ref{D23Red}.
 
  \noindent $\bullet$ \; $F_3\cup F_4  \cup {\mathcal S}_1\cup {\mathcal S}_2$ are the one dimensional strata.
  
   \noindent $\bullet$ \; $p$ is the 
   zero
   dimensional stratum.

 \item[vi)]  The behavior of the  principal foliations  \fp 1 and \fp 2 in the
neighborhood of $ p $ is as illustrated  in Fig. \ref{D23S} (left  and right).
\end{itemize}
\end{theorem}

\begin{figure}[h]
\begin{center}
   \def\svgwidth{0.9\textwidth}
    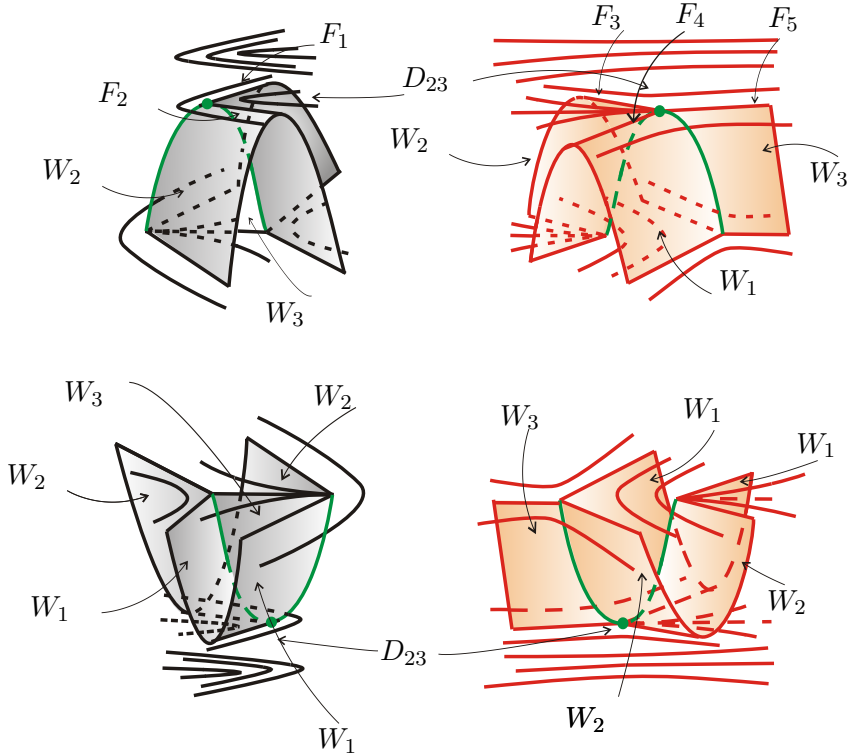
    \caption{Principal foliations  $\mathcal{F}_1(\alpha)$ (Left, views from  top and bottom) and  $\mathcal{F}_2(\alpha)$ (Right, views from top and bottom) in the neighborhood of  a   point $D_{23}$.  }
  \label{D23S}
    \end{center}
\end{figure}

 \begin{figure}[h]
\begin{center}
    \def\svgwidth{1.0\textwidth}
    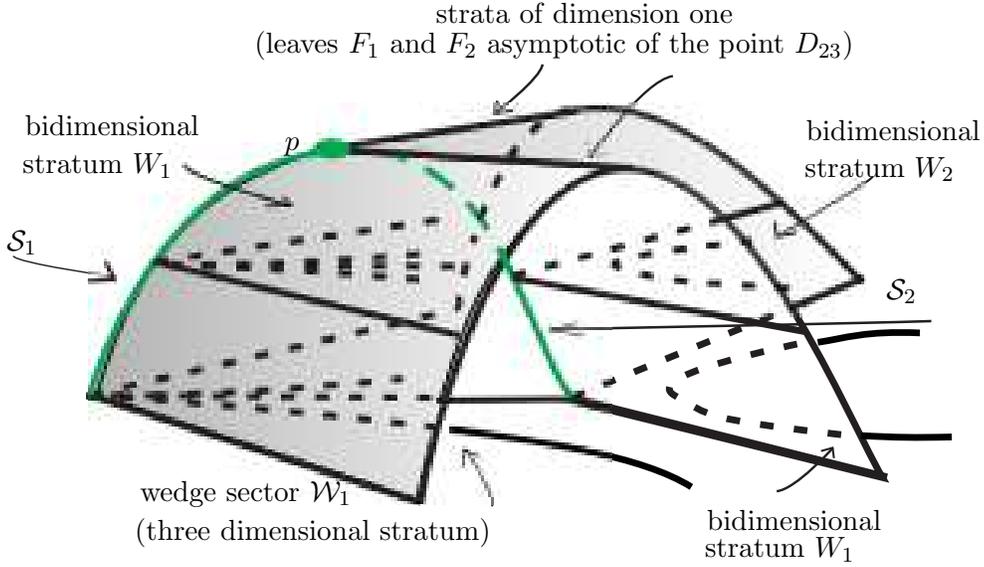
    \vspace{.3cm}
    \caption{Stratification of the wedge sector (strata of dimensions three, two and  one) and umbilic separatrices of an arc $D_2 - D_{23} - D_3$ of ${\mathcal F}_{1}(\alpha)$. }\label{D23Black}
       \end{center}
       \end{figure}

       \begin{figure}[h]
       \begin{center}
          \def\svgwidth{0.7\textwidth}
    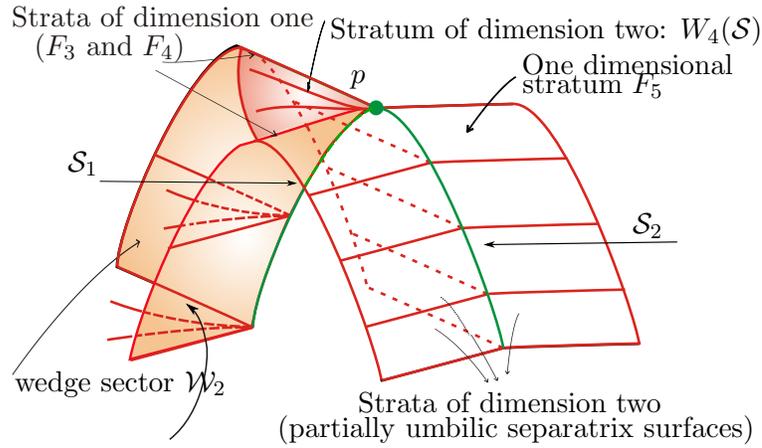
   \caption{  Stratification of the wedge sector (strata of dimensions three, two and  one) and partially umbilic separatrix surfaces of an arc $D_2  - D_{23}  - D_3$ of ${\mathcal F}_{2}(\alpha)$. } \label{D23Red}
    \end{center}
\end{figure}

\begin{remark} \label{rem:chi23}
The condition $\chi_{23}\neq0$  is equivalent to  the quadratic contact, at $D_{23}$,  of the partially umbilic 
curve $\mathcal S_{12}$ with the plane field $\mathcal P_3$.

The analysis in section \ref{sec:d23} shows that  condition $\chi_{23}\neq0$   gives also the {\emph saddle - node} character of an equilibrium point which
is essential  to establish the principal configuration around the point $D_{23}$, exhibiting 
the transition between the types $D_2$ and
$D_{3}$. 
\end{remark}

\begin{remark}\label{rem:indep_coord}

\noindent   The conditions that define the types $D_i$ and
$D_{ij}$ are independent of  the  coordinate charts.
 This can be verified through  a direct
calculation, similar to that
performed
in the two dimensional case, \cite{gs1,  gs5}.
 The conditions above  that define in  sections 
 \ref{defDi}, \ref{defD12} and \ref{defD23}
  the partially umbilic points studied in this work
 are
closely related to those that define the {\it Darbouxian and
Semi-Darbouxian umbilic points} in the case of  two dimensional surfaces, see
\cite{gs4} and \cite{gs1, gs5, gs6}.
              \end{remark}

\section {Differential Equations of the Lines of Principal Curvature and  Lie-Cartan Suspension }\label{sec:2}

In this section will be obtained the differential equation of the principal line fields \lp i.

Let $\alpha:\func M3{\mathbb R} 4$ be an immersion of  class $C^k,\; k\ge 6$,
where $M^3$
is compact and oriented manifold. The space ${\mathbb R}^4$ is also oriented
with the orientation fixed by the canonical basis $\{ E_1, E_2,
E_3, E_4\} $.

In a local chart  $\uu :\func M3{\mathbb R} 3$
the first and the second fundamental forms associated to  the
immersion  $\alpha$ are given,  respectively, by
$ \;I_\alpha = \sum \gf ij d\ui i d\ui j\; $ and $ \; II_\alpha = \sum \la ijd\ui
i d\ui j,$  where
$\gf ij = \langle\pa i,\pa j\rangle \; \text{and}\; \la ij = \langle\pad ij,\mathcal N
\rangle$ and  $ \mathcal N $ is the positive normal vector to the immersion
 $\alpha$ which is defined by:
 $ \; {\mathcal N}_{\alpha}(p) = (\alpha_1 \we \alpha_2 \we \alpha_3)(p)/|
(\alpha_1 \we \alpha_2 \we \alpha_3)(p)|$

The normal curvature at the point $ p $ in the direction
 $ \; v
= (d\ui 1, d\ui 2, d\ui 3)\quad $  is defined by
$k_n(p)(v)
= (\frac{ II_{\alpha}}{I_{\alpha}})(p)(v,v)$.

The directions at which $k_n $ assume critical  values
are the {\it principal directions} and the
values of $k_n $ in these directions are the
{\it principal curvatures}. They
will be denoted by  $ \fu k1(p) \leq  \fu k2(p)
\leq \fu k3(p) $.
 
To calculate the {\em principal directions}
observe that the normal curvature can also  defined by
 $k_n(p)= II_\alpha
$
subject to the
condition $\fu I\alpha=1.$
Therefore using the method of Lagrange
with  parameter $ k,$ it results:

$$\aligned \left(\begin{matrix}
\la 11&\la 12&\la 13\\\la 12&\la 22&\la 23\\\la 13&\la 23&\la
33\end{matrix}\right)  \left(\begin{matrix} d\ui 1\\d\ui 2\\d\ui 3\end{matrix}\right)  &=k
\left(\begin{matrix}\gf 11&\gf 12&\gf
13\\\gf 12&\gf 22&\gf 23\\\gf 13&\gf 23&\gf 33\end{matrix}\right) \left(\begin{matrix}
d\ui 1\\d\ui 2\\d\ui 3\end{matrix}\right)
\endaligned
$$

\noindent  Therefore  the principal directions are given by the following system
of equations:
 \begin{equation} \label{eq:1}
 \aligned
(\la 11-\fu ki\gf 11)du_1 + (\la 12 - \fu ki\gf 12)du_2 + (\la 13-\fu
 ki\gf 13)du_3 &= 0   \\
(\la 12-\fu ki\gf 12)du_1 + (\la 22 - \fu ki\gf 22)du_2 + (\la 23-\fu
 ki\gf 23)du_3 &= 0\\
(\la 13-\fu ki\gf 13)du_1 + (\la 23 - \fu ki\gf 23)du_2 + (\la 33-\fu
 ki\gf 33)du_3 &= 0  \endaligned \end{equation}

\noindent  where $\fu ki~~ (i=1,2,3)$  are the principal  curvatures,
which are  defined by the equation $ \text{det}(\la ij - k \gf ij) = 0
$.
 This equation is useful only when the correspondent  principal curvature is smooth. Near the partially umbilic points two   principal curvatures are only continuous and it  is more convenient to consider implicit differential equations.

 Next  will be obtained  the Lie-Cartan vector field that will be
used for the analysis of principal curvature line near the
 partially umbilic points in this work.

Consider the plane passing through $q\in M$ having the principal direction $e_3(q)$ as normal vector:

\begin{equation}\label{campoplanos}
\mathcal P_{3}(q)=\{(du_1,du_2,du_3);\langle(du_1,du_2,du_3),G\cdot (e_3(q))^T\rangle=0\},
\end{equation}
where $G=[g_{ij}]_{3\times3}$ is the matrix of the first fundamental form.

Using equation \eqref{eq:1},
solving the linear system
 it follows
 that the principal direction
 $e_3(q)=(du_1,du_2,du_3)$ is
 given by

$$\aligned
\frac{du_1}{du_3}\; &=\;\frac{U_1(u_1,u_2,u_3)}{W_1(u_1,u_2,u_3)},\;\;\;\;
\frac{du_2}{du_3}\; &=\; \frac{V_1(u_1,u_2,u_3)}{W_1(u_1,u_2,u_3)}
\endaligned $$

\noindent where
\begin{equation}\label{eq:UVW1}
\aligned
U_1&=\left(g_{12}g_{23}-g_{22}g_{13} \right)k_3^{2}+ \left(-g_{12}\lambda_{23}-g_{23}\lambda_{12}+\lambda_{22}g_{13}+g_{22}\lambda_{13} \right) k_3+\\ &+\lambda_{23}\lambda_{12}-\lambda_{22}\lambda_{13}\\
V_1&=\left(-g_{11}g_{23}+g_{13}g_{12} \right)k_3^{2}+ \left(\lambda_{11}g_{23}+g_{11}\lambda_{23}-\lambda_{13}g_{12}-g_{13}\lambda_{12} \right)k_3+\\
&-\lambda_{11}\lambda_{23}+\lambda_{13}\lambda_{12}\\
W_1&=\left(g_{11}g_{22}-g_{12}^{2} \right)k_3^{2}+ \left(-\lambda_{11}g_{22}-g_{11}\lambda_{22}+2\lambda_{12}g_{12}\right)k_3+\\
&+\lambda_{11}\lambda_{22}-\lambda_{12}^{2}.
\endaligned
\end{equation}

Notice  that   $W_1\neq0$ in a neighborhood of the partially umbilic point. This follows from
the calculations  displayed in subsections \ref{ssec:AfirstFF},  \ref{ssec:AsecondFF} and \ref{ssec:Ak3} of  the Appendix  \ref{sec:App},
which give $W_1(0,0,0)=(k-k_3)^2$.

Therefore, from equation (\ref{campoplanos}),  the  field of planes  $\mathcal P_{3} $ is defined by the
field of
 kernels of
 the 
differential one form
\begin{equation}\label{eqcampoplanos}
 \aligned
\omega=&
[g_{11}{U_1}+ g_{12}V_1+g_{13}{W_1} ]du_1 +
[g_{12}{U_1}+g_{22}{V_1}+
g_{23}{W_1} ] du_2 \\
 +&
[g_{13} {U_1} +  g_{23}V_1 +g_{33} {W_1}] du_3=
0
\endaligned
\end{equation}

\begin{remark} \label{rem:nonfrob}
The plane field $\mathcal P_{3} $ is in general not integrable.
  Therefore the
 analysis of the integral foliations $\mathcal{F}_1$  and $\mathcal{F}_2$  of the line fields $\mathcal L _1$
 and  $\mathcal L _2$ near a partially umbilic curve $\mathcal S_{12}$
 is strictly three-dimensional. In other words it cannot be reduced
  to a family of two-dimensional principal configurations on the integral leaves of  the
  field of planes $\mathcal P_{3}$,  obtained from  Frobenius  Theorem. See Spivak \cite{sp2}.

 In fact, calculation in the Monge chart in  equation \eqref{eq:di},
 gives that  $\omega\wedge d\omega$ is expressed   near $0$  by:

$$(\omega\wedge d\omega) =-{(k-k_3)^{2}}[
  q_{111}\left(a- b \right) u_1+
  \left(b\, (q_{021}- q_{201}) - c \,q_{111} \right) u_2+h.o.t.]
$$

This shows that
the condition for  Frobenius integrability $\omega\wedge d\omega = 0 $, identically,  does not hold generically.
\end{remark}

The principal directions $e_1(q)$ and $e_2(q)$
associated to $k_1(q)$ and $k_2(q)$ belong to the plane $\mathcal P_{3}(q)$.

Consider the  first and second fundamental forms of $\alpha$ restricted to $\mathcal P_{3}$ and write,

\begin{align*}
I_r(du_1,du_2)&=I_{\alpha}\Big|_{du_3=\mathcal{U}(u_1,u_2,u_3)du_1+\mathcal{V}(u_1,u_2,u_3)du_2}\\
&=E_rdu_1^2+2F_rdu_1du_2+G_rdu_2^2,\\
II_r(du_1,du_2)&=II_{\alpha}\Big|_{du_3=\mathcal{U}(u_1,u_2,u_3)du_1+\mathcal{V}(u_1,u_2,u_3)du_2}\\
&=e_rdu_1^2+2f_rdu_1du_2+g_rdu_2^2,
\end{align*}

\noindent where
$$\mathcal{U}=-\frac{[g_{11}U_1+ g_{12}{V_1}+g_{13} {W_1}]}{[g_{13}{U_1}+  g_{23}{V_1}+g_{33}{W_1}]}, \; \mathcal{V}=-\frac{[g_{12}{U_1} +g_{22}{V_1}+
g_{23}{W_1}]}{[g_{13}{U_1}+  g_{23}{V_1}+g_{33}{W_1}]}, $$

$$E_r=\frac{\partial^2 I_r}{2\partial du_1^2}(0,0), \; F_r=\frac{\partial^2 I_r}{2\partial du_1\partial du_2}(0,0) \; G_r=\frac{\partial^2 I_r}{2\partial du_2^2}(0,0),$$

$$ e_r=\frac{\partial^2 II_r}{2\partial du_1^2}(0,0),  \; f_r=\frac{\partial^2 II_r}{2\partial du_1\partial du_2}(0,0),  \; g_r=\frac{\partial^2 II_r}{2\partial du_2^2}(0,0).$$

Write $k_n^r(q;du_1,du_2)=\m\frac{II_r}{I_r}(q,du_1,du_2)$,
where  $I_r(q)$ and $II_r(q)$ are the first and second fundamental forms of  $\alpha$ restricted to the plane $\mathcal P_{3}(q)$.

Let $P=\m\frac{du_2}{du_1}$. Therefore
the slopes of the
principal directions  $e_1(q)$ and $e_2(q)$ in the plane $\mathcal P_{3}$ are defined by the implicit differential  equation:
\begin{equation}\label{eqLC}
\aligned
{\mathcal L}=& L_r(u_1,u_2,u_3)P^2+M_r(u_1,u_2,u_3)P+N_r(u_1,u_2,u_3)=0\\
\omega=&[g_{11}{U_1}+ g_{12}V_1+g_{13}{W_1} ]du_1 +
[g_{12}{U_1}+g_{22}{V_1}+
g_{23}{W_1} ] du_2 \\
 +&
[g_{13} {U_1} +  g_{23}V_1 +g_{33} {W_1}] du_3=
0,   \endaligned
\end{equation}

 \noindent where
 
\begin{equation} \label{eq:restr_LMN}
L_r=F_{r}g_{r}-f_{r}G
_{r}, \quad M_r=E_{r}g_{r}-e_{r}G_{r},\quad N_r=E_{r}f_{r}-e_{r}F_{r}.
\end{equation}

\begin{remark}\label{remarkPU}
The partially umbilic  points {\rm (}$k_1=k_2${\rm)} are defined by $L_r(u_1,u_2,u_3)=0$ and $M_r(u_1,u_2,u_3)=0$.
\end{remark}

Let $\mathcal{L}(u_1,u_2,u_3,P)=L_r(u_1,u_2,u_3)P^2+M_r(u_1,u_2,u_3)P+N_r(u_1,u_2,u_3)$.

Consider in the $(u_1,u_2,u_3,P)$-space the hypersurface
\begin{equation}\label{hipLieCartan}
{\mathcal L}=\{(u_1,u_2,u_3,P):\;\mathcal{L}(u_1,u_2,u_3;P)=0\},
\end{equation}
called {\em Lie-Cartan hypersurface}.

\begin{remark} 
   To cover the whole  sub-bundle of the tangent projective 
   bundle over M, 
  defined by the  lines contained   in   the
field of planes orthogonal to $\mathcal P _3$    one  
 must consider also the equation 
 $L_r du^2_2 + M_r du_1du_2 + N_r du^2_1 = 0$, with
the projective coordinate $ Q = du_1/du_2$. 
Direct calculation shows that, in the cases considered in this paper,  there 
are 
no singularities to analyze near the points with  $Q=0$  which represent the 
points  $P=\infty$.
\end{remark}

\begin{proposition}\label{prop:1}  The vector field
$
X=X_{\mathcal L}= X_1\frac{\partial}{\partial u_1}+X_2\frac{\partial}{\partial u_2}+
X_3\frac{\partial}{\partial u_3}
+X_4\frac{\partial}{\partial P}
$
where
\begin{equation}\label{coordLCVF}
\begin{array}{lll}
X_1&=&\mathcal{L}_P\\
X_2&=&P\mathcal{L}_P\\
X_3&=&\left(\mathcal{U}+\mathcal{V}P\right)\mathcal{L}_P\\
X_4&=&-\left(\mathcal{L}_{u_1}+P\mathcal{L}_{u_2}+\mathcal{L}_{u_3}\left(\mathcal{U}+\mathcal{V}P\right)\right)
\end{array}
\end{equation}

\noindent  is of class $C^{k-3}$, tangent to $\mathcal L$ and the projections
of the integral curves of $X$ by $\pi(u_1,u_2,u_3,P)=(u_1,u_2,u_3)$ are  the principal lines of
the two principal foliations ${\mathcal F}_1$ and ${\mathcal F}_2$
which are singular
along
the  partially umbilic curve ${\mathcal S}_{12}$.
\end{proposition}

\begin{proof} By direct verification the
vector field $X=X_{\mathcal L}$ is
 tangent to the hypersurface  ${\mathcal L}^{-1}(0)$ and its integral curves
 project to solutions of the implicit differential equation \eqref{eqLC}.
\end{proof}

\section{Proof of Theorem \ref{th:d123}}\label{sec:4}

\noindent {\bf Sketch of Proof.}  The behavior of the  curvature lines  near a partially umbilic curve $\mathcal S$ will be described  from the
analysis of the lifted vector field
$X  = \fu X1 \partial/\partial u_1   + \fu X2\partial/\partial u_2  + \fu  X3\partial/\partial u_3   + \fu  X4\partial/\partial P $
given in equation  \eqref{coordLCVF}.

The projections of the integral curves of $X$   are  the curvature lines outside
 the umbilic and partially umbilic  points.  It will be shown that $X$  has
 lines of singularities contained in $\mathcal L$.
More precisely, in the $D_1$ case, $X$ has a line  of
singularities
 $\beta_1$ normally hyperbolic
such that the stable and unstable manifolds,
$ W^s_X(\beta_1)=\{p :\omega(p)=\beta_1\}$ and $ W^u_X(\beta_1)=\{p:\alpha(p)=\beta_1\} $
 are two dimensional smooth surfaces with $\partial W^s_X(\beta_1)=\partial W^u_X(\beta_1)=\beta_1$.

In the $D_2$ case, $X$ has three lines of singularities, $\beta_1$, $\beta_2$
and $ \beta_3$, normally hyperbolic and such that dim$W^u_X(\beta_1)=3$ (
repeller normally hyperbolic) and near $\beta_2$ and $\beta_3$ the behavior of
$X$ is as in $D_1$ case.

Finally, in  $D_3$ case, $X$ has  three lines of singularities and
near each line the behavior of $X$ is as in the $D_1$ case. See Figs. \ref{projD1},
\ref{projD2} and \ref{projD3}.

The partially umbilic separatrices are the diffeomorphic images  of the bi-dimensional invariant manifolds associated to the normally hyperbolic singularities of $X$.

\subsection{ Local analysis of the Lie-Cartan vector field $X  = \fu X1 \partial/\partial u_1   + \fu X2\partial/\partial u_2  + \fu  X3\partial/\partial u_3   + \fu  X4\partial/\partial P $ }\label{sub:41}

\noindent First  the singularities of $X$ will be determined.

In the Monge chart, see equation \eqref{eq:di}  and
 the
 Appendix \ref{sec:App},  \ref{ssec:firstP3} and \ref{ssec:secondP3} where the  defining restricted functions  
 in equation (\ref{eq:restr_LMN})  are displayed, 
 the vector field $X=X_{\mathcal L}$ given  in equation \eqref{coordLCVF} is written  as:

\begin{equation}\label{CLCcap1}
\begin{array}{ll}
X_1&=(-2bu_2-2q_{111}u_3 )P+(-a+b)u_1+cu_2+(-q_{201}+q_{021})u_3+h.o.t\\
X_2&=PX_1\\
X_3&=\left(-{\m\frac {q_{111}u_1+q_{021}u_2+q_{012}u_3}{k-k_3}}P-{\m\frac {q_{201}u_1+q_{111}u_2+q_{102}u_3}{k-k_3}}+h.o.t
\right)X_1\\
X_4&=A_3(u_1,u_2,u_3)P^3+A_2(u_1,u_2,u_3)P^2+A_1(u_1,u_2,u_3)P+A_0(u_1,u_2,u_3)
\end{array}
\end{equation}
where
{\small
\begin{align*}
A_3(u_1,u_2,u_3)&=b+\left(C-{k}^{3}+{\frac {q_{111}^{2}+q_{201}q_{021}}{k-k_3}} \right)u_1+ \left(D+3{\frac {q_{111}q_{021}}{k-k_3}} \right) u_2+\\
&+\left(Q_{121}+{\frac {2q_{111}q_{012}+q_{102}q_{021}}{k-k_3}}\right)u_3+h.o.t\\
A_2(u_1,u_2,u_3)&=-c+ \left(-D+2B+{\frac {6q_{111}q_{201}-3q_{111}q_{021}}{k-k_3}} \right)u_1+\\
&+ \left(-E+{k}^{3}+2C+{\frac {4q_{111}^{2}-3q_{021}^{2}+2q_{201}q_{021}}{k-k_3}} \right)u_2+ \\
&+\left(-Q_{031}+2{\frac {Q_{211}\left(2q_{201}q_{012}+4q_{102}q_{111}-3q_{012}q_{021} \right) }{k-k_3}} \right)u_3+h.o.t
\end{align*}
}

{\small
\begin{align*}
A_1(u_1,u_2,u_3)&=a-2b+ \left(-2C+A-{k}^{3}+{\frac {-2q_{201}q_{021}-4q_{111}^{2}+3q_{201}^{2}}{k-k_3}} \right)u_1+\\
&+\left(-2D+B+{\frac {3q_{111}q_{201}-6q_{111}q_{021}}{k-k_3}} \right)u_2+ \\
&+\left(-2Q_{121}+Q_{301}+{\frac {3q_{102}q_{201}-2q_{102}q_{021}-4q_{111}q_{012}}{k-k_3}} \right)u_3+h.o.t\\
A_0(u_1,u_2,u_3)&=\left(-B-3{\frac {q_{111}q_{201}}{k-k_3}} \right)u_1+ \left(-C+{k}^{3}-{\frac {2{q_{111}}^{2}+q_{201}q_{021}}{k-k_3}} \right)u_2+\\
&+ \left(-Q_{211}-{\frac {q_{201}q_{012}+2q_{102}q_{111}}{k-k_3}} \right)u_3+h.o.t
\end{align*}

 }
\begin{lemma}\label{lema1}
Let $\mathcal S$ be a Darbouxian partially umbilic curve.
\begin{itemize}
\item[i)] If  $\mathcal S$  is of type  $D_1$,  then  $X$ has a line of singularities $\gamma_1$.
\item[ii)] If $\mathcal S$ is of type $D_2$ or  $D_3$ then  $X$ has three lines of singularities $\gamma_i$, $(i=1,2,3)$.
 \end{itemize}

\end{lemma}

\begin{proof}
The singular points of $X=X_{\mathcal L}$ are given by\\
\begin{equation*}
\left\{
\begin{array}{l}
\left\{\begin{array}{c}
L_r(u_1,u_2,u_3)=0,\\
M_r(u_1,u_2,u_3)=0,
\end{array}\right.
\textrm{ (Partially Umbilic Points, see remark \ref{remarkPU})}\\
\!\!\!A_3(u_1,u_2,u_3)P^3+A_2(u_1,u_2,u_3)P^2+A_1(u_1,u_2,u_3)P+A_0(u_1,u_2,u_3)=0
\end{array}
\right.
\end{equation*}

As $\det\left(\m\frac{\partial(L_r,M_r)}{\partial(u_1,u_2)}\Big|_{u_1=0=u_2=u_3}\right)=b(b-a)\neq 0$ we can write $u_1=u_1(u_3)$ and $u_2=u_2(u_3)$ in $L_r(u_1,u_2,u_3)=0$ and $M_r(u_1,u_2,u_3)=0$.

Let $C(u_3,P)=A_3(u_1(u_3),u_2(u_3),u_3)P^3+
A_2(u_1(u_3),u_2(u_3),u_3)P^2+A_1(u_1(u_3),u_2(u_3),u_3)P+A_0(u_1(u_3),u_2(u_3),u_3)$.
 The discriminant of
 $C(u_3,P) =0$
 is given by
\begin{equation}\label{discrd3}
D(u_3)=-\left(\frac{c^2}{4b^2}-\frac{a}{b}+2\right)\left(\frac{2a}{b}-4\right)^2+O(1),
\end{equation}
Therefore, for $u_3$ small  it follows that:
\begin{align*}
&\textrm{Condition }D_1)\;\; \Rightarrow\;\;  D(u_3)>0\Rightarrow
\textrm{There exists an unique solution $P_1(u_3)$}\\
&\textrm{of }C(u_3,P)=0\\
&\textrm{Conditions }D_2 \textrm{ and } D_3 )
\;\; \Rightarrow\;\;   D(u_3)<0\;\; \Rightarrow\;\;  \textrm{There exist exactly $3$ solutions}\\
&\textrm{of }C(u_3,P)=0\textrm{ that will be denoted by $P_1(u_3)$, $P_2(u_3)$ and $P_3(u_3)$}.
\end{align*}
As the equation $C(0,P)=0$ has the solutions:
 
\begin{align*}P_1(0)=0;P_2(0)=\frac{c}{2b}+\sqrt{\frac{c^2}{4b^2}-\frac{a}{b}+2},
P_3(0)=\frac{c}{2b}-\sqrt{\frac{c^2}{4b^2}-\frac{a}{b}+2},
\end{align*}
it follows that:
\begin{itemize}
\item[$\bullet$] If  $\mathcal S$ is of type  $D_1$ then there is the unique
 curve of singularities of $X$. The intersection of this curve with the axis  $P$ is the point $P_1(0)$,

\item[$\bullet$] If $\mathcal S$ is of type $D_2$ or $D_3$ then $X$ has   exactly
 $3$ curves of singularities
and these curves cross the axis  $P$ at $P_1(0)$, $P_2(0)$ and  $P_3(0)$.
\end{itemize}

For $i=1,2,3$, denote the curve of singularities of   $X$,
intersecting the axis  $P$ at $P=P_i(0)$ by $\gamma_i(u_3)$.

\end{proof}

\begin{lemma}\label{lema2}
Let $X=X_{\mathcal{L}}$ be the  Lie-Cartan vector field restricted to
the hypersurface
$\mathcal{L}(u_1,u_2,u_3,P)=0.$

 Then:

\begin{itemize}
 \item[i)] {Condition } $D_1$ \;$\;\; \Rightarrow\;\;   \;\;\gamma_1$ {is normally hyperbolic of saddle type }
of $X_{\mathcal{L}}$,

\item[ii)] {Condition} $D_2$ $ \;\;\Rightarrow\;\;  $ \;\;  {for $i=1,2,3$, $\gamma_i$  is normally
 hyperbolic of $X_{\mathcal{L}}$ and satisfy: one is attractor(or repeller) and the other two are of saddle type},

\item[iii)] {Condition} $D_3$ $ \;\; \Rightarrow\;\; $  \;\; for $i=1,2,3$, $\gamma_i$  { is normally hyperbolic of saddle
 type of } $X_{\mathcal{L}}.$

\end{itemize}

\end{lemma}

 \begin{proof}
Let $X$ be the Lie-Cartan vector field.
The
 linearisation of $X$ near the singularities will be analyzed
 below.

   Let
\begin{equation}\label{Betaiparametrizada}
\gamma_i(u_3)=(c_1(u_3),c_2(u_3),u_3,P_i(u_3)),\quad i=1,2,3,
\end{equation}
the curves of singularities given in lemma \ref{lema1}.

The characteristic polynomial of  $DX(\gamma_i(u_3))$ is
\begin{equation}\label{PolCarac}
p(\lambda)=\lambda^2\cdot\left(\frac{\partial X_4}{\partial P}-
\lambda\right)\cdot\left(\lambda-\frac{\partial X_3}{\partial u_3}-
\frac{\partial X_2}{\partial u_2}-\frac{\partial X_1}{\partial u_1}\right)
\end{equation}
\begin{itemize}
\item Condition  $D_1(\frac{a}{b}>\left(\frac{c}{2b}\right)^2+2)$:
\end{itemize}
Suppose that $b>0$. The case $b<0$ is similar. From equations (\ref{PolCarac}) and (\ref{CLCcap1}) it follows that the
 eigenvalues of $DX(\gamma_1(u_3))=DX(c_1(u_3),c_2(u_3),u_3,P_1(u_3))$
are $\lambda_1(u_3)\equiv\lambda_2(u_3)=0$ and
\begin{equation}\label{autovaloresD1}
\begin{array}{ll}
\lambda_3(u_3)&=\m\frac{\partial X_4}{\partial P}=-2b+a+O(u_3),\\
&\\
\lambda_4(u_3)&=\m\frac{\partial X_1}{\partial u_1}+
\m\frac{\partial X_2}{\partial u_2}+\m\frac{\partial X_3}{\partial u_3}=b-a+O(u_3),
\end{array}
\end{equation}

For $u_3$ small, $\lambda_3>0$ and $\lambda_4<0$.
\begin{itemize}
\item Condition $D_2$ ($1<\frac{a}{b}<\left(\frac{c}{2b}\right)^2+2$):
\end{itemize}
In this case, $DX(\gamma_1(u_3))$  has eigenvalues
 $\lambda_1=\lambda_2\equiv0$ and $\lambda_3$, $\lambda_4$ defined by equation (\ref{autovaloresD1}). In this case, for  $u_3$ small it follows that:
\begin{equation}\label{lambda1D2}
\lambda_4(u_3)<0\textrm{ and}
\left\{\begin{array}{c}
\lambda_3(u_3)>0\textrm{ if } a>2b\\
\lambda_3(u_3)<0\textrm{ if } a<2b
\end{array}\right.
\end{equation}

Suppose that  $a<2b$ and $b>0$. The other cases can be considered similarly. For $i=2,3$ the eigenvalues $\lambda_1$, $\lambda_2$, $\lambda_3$ and $\lambda_4$
of $DX(\gamma_i(u_3))=DX(c_1(u_3),c_2(u_3),u_3,P_i(u_3))$ satisfy:
 $\lambda_1=\lambda_2\equiv 0$, $\lambda_3>0$ and $\lambda_4<0$ since for
equations (\ref{PolCarac})  and (\ref{CLCcap1}) it follows that:

\begin{equation}\label{lambda3e4D2}
\lambda_3(u_3)=bP_i(0)^2+2b-a+O(1)>0, \lambda_4(u_3)=-bP_i(0)^2-b+O(1)<0
\end{equation}

For $a<2b$ it follows that $P_2(u_3)<0<P_3(u_3)$, for  $u_3$ small.

\begin{itemize}
\item Condition $D_3$ ($\m\frac{a}{b}<1$):
\end{itemize}
The eigenvalues of  $DX(\gamma_i(u_3))$, for  $i=1,2,3$, satisfy:
\begin{equation}\label{autovaloresD3}
\lambda_1(u_3)=\lambda_2(u_3)\equiv 0\textrm{ and }\lambda_3(u_3)\cdot\lambda_4(u_3)<0
\end{equation}
since $\m\frac{a}{b}<1\Rightarrow\m\frac{a}{b}<2$, see (\ref{lambda1D2}) and (\ref{lambda3e4D2}).

In the Monge chart we have that

\begin{align*}
\frac{\partial\mathcal{L}}{\partial u_2}(0,0,0,P_i(0))&=\frac{\partial L_r}{\partial u_2}(0,0,0)(P_i(0))^2+\frac{\partial M_r}{\partial u_2}(0,0,0)P_i(0)+\frac{\partial N_r}{\partial u_2}(0,0,0)\\
&=-b(P_i(0))^2+cP_i(0)+b=\left\{
\begin{array}{ll}
b, &\textrm{ if } i=1,\\
a-b, &\textrm{ if } i=2,3.
\end{array}\right.
\end{align*}

Therefore in a neighborhood of $(0,0,0,P_i(0))$ ($i=1,2,3$),  we can write $u_2=u_2(u_1,u_3,P)$ in the equation $\mathcal{L}(u_1,u_2,u_3,P)=0$. In the chart $(u_1,u_3,P)$ the  Lie-Cartan vector field is given by
\begin{equation}\label{CLCRcap1}
X_{\mathcal{L}}=\left\{\begin{array}{ll}
\dot u_1&=X_1(u_1,u_2(u_1,u_3,P),u_3,P)\\
\dot u_3&=X_3(u_1,u_2(u_1,u_3,P),u_3,P)\\
\dot P&=X_4(u_1,u_2(u_1,u_3,P),u_3,P)
\end{array}\right.
\end{equation}

The linearisation of  $X_{\mathcal{L}}$ in $(0,0,0,P_i(0))$ has one eigenvalue equal to zero and the other two are non zero
and satisfy the same conditions of  $\lambda_3$ and $\lambda_4$, given in equation
  (\ref{autovaloresD1}), (\ref{lambda1D2}), (\ref{lambda3e4D2}) and (\ref{autovaloresD3}).
The eigenvector associated to $(0,0,0,P_i(0))$ is tangent to the curve of singularities.
Therefore  it follows that:
\begin{align*}
\textrm{Condition } D_1)\; \Rightarrow\;\;  & \gamma_1(u_3) \textrm{ is normally hyperbolic of saddle type,}\\
\textrm{Condition } D_2)\; \Rightarrow\;\; & \textrm{if }a<2b\textrm{ and }b>0\textrm{
it follows that for $u_3$ sufficiently small}\\
&\begin{array}{l}
\!\!\!\textrm{the curve }\gamma_1(u_3)\textrm{ is normally hyperbolic of attracting type }\\
\!\!\!\textrm{the curves }\gamma_2(u_3)\textrm{ and }\gamma_3(u_3)\textrm{ are normally hyperbolic of }\\
\!\!\!\textrm{saddle type. The other cases are similar. }
\end{array}\\
\textrm{Condition } D_3)\; \Rightarrow\;\;  & \gamma_i(u_3)(i=1,2,3)\textrm{
are normally hyperbolic of saddle type}.
\end{align*}
\end{proof}

\subsection{End of proof of Theorem \ref{th:d123} }

Let $\mathcal S$ be a Darbouxian partially umbilic curve. Suppose  that $\mathcal S$ is of type $D_1$.
By lemma \ref{lema1} there exists a
 unique curve of singularities,  $\gamma_1(u_3)$ (see (\ref{Betaiparametrizada})), of the Lie-Cartan vector field. By lemma \ref{lema2}, $\gamma_1(u_3)$ is normally hyperbolic of saddle type.

By Invariant Manifold Theory (see \cite[page 44]{hps} and \cite{fe})   in a neighborhood $V_{\gamma_1}$ of $\gamma_1(u_3)$ there are two dimensional invariant manifolds  $W^{s}_{\gamma_i}$ and $W^{u}_{\gamma_1}$, of class $C^{k-3}$, with
 $W^{u}_{\gamma_1}\cap W^{s}_{\gamma_1}=\gamma_1$.

\textit{Claim:} $\Pi(W^{u}_{\gamma_1})=\mathcal S$, where $\Pi(u_1,u_2,u_3,P)=(u_1,u_2,u_3)$. In fact,
\begin{itemize}
\item the axis $P$ is  invariant by $X$
\item $(0,0,0,1)$ is the eigenvector associated to the eigenvalue
 $\m\frac{\partial X_4}{\partial P}(\gamma_1(u_3))$ which is positive for $u_3$ sufficiently small, see  equation
 (\ref{autovaloresD1}).
\end{itemize}
By
the
 uniqueness   of the invariant manifolds, it follows that in a neighborhood of
 $\gamma_1$, $W^u_{\gamma_1}=\mathcal{S}\times(\textrm{axis }P)$, and therefore $\Pi(W^{u}_{\gamma_1})=\mathcal S$.
Let
$$
V_{\mathcal{S}}=\Pi(V_{\gamma_1})\textrm{ and  }W(\mathcal{S})=\Pi(W^{s}_{\gamma_1}).
$$

Therefore it follows that: if  $\mathcal S$  is of type $D_1$ then
\begin{itemize}
\item there exists a  unique umbilic separatrix surface, $W({\mathcal S})$, of class $C^{k-3}$,
\item which  is fibred over $\mathcal S$ and the fibers are the  leaves of  $\mathcal{F}_{1}(\alpha)$,
\item there exists a tubular neighborhood  $V_{\mathcal{S}}$ of  $\mathcal S$  such that the set $V_{\mathcal{S}}\setminus W({\mathcal S})
$ is a hyperbolic
sector of   $\mathcal{F}_{1}(\alpha)$.
\end{itemize}

See Fig. \ref{projD1}.

\begin{figure}[h]
\psfrag{pi}{$\Pi$}
\psfrag{pws1}{$W_1({\mathcal S)}= \Pi(W^s_{\gamma_1} )$}
\psfrag{ws1}{$ W^s_{\gamma_1} $}
\psfrag{beta1}{$\gamma_1$}
\psfrag{wu1}{$  W^u_{\gamma_1} $}
\psfrag{s}{$\mathcal S$}
\begin{center}
    \def\svgwidth{0.7\textwidth}
     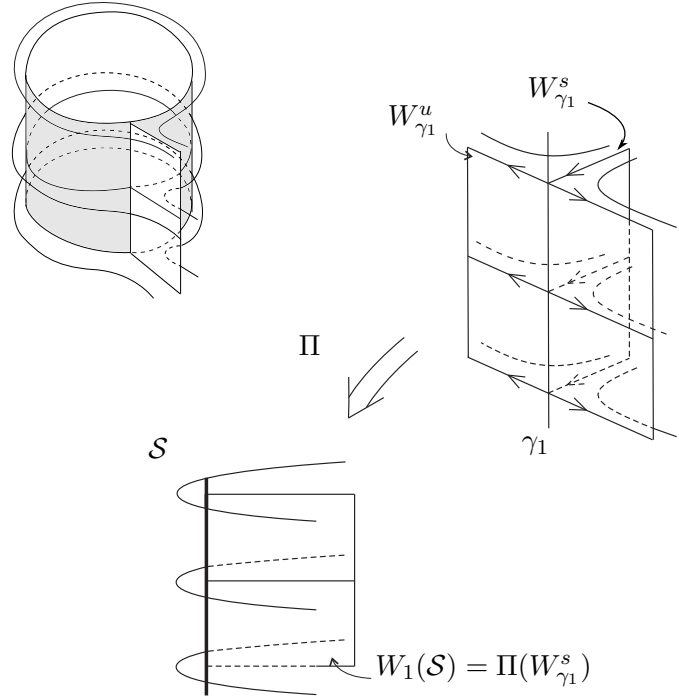
    \caption{Lie-Cartan resolution of a partially umbilic point $D_1$.}
  \label{projD1}
    \end{center}
\end{figure}

Suppose  that $\mathcal S$ is of type $D_2$. By lemma \ref{lema1}, there exist
three curves of singularities $\gamma_1(u_3)$, $\gamma_2(u_3)$ and $\gamma_3(u_3)$ of the Lie-Cartan vector field. By lemma \ref{lema2}, $\gamma_1(u_3)$ is normally hyperbolic (attractor);  $\gamma_2(u_3)$ and
 $\gamma_3(u_3)$ are normally hyperbolic of saddle type.

By Invariant Manifold Theory  (see \cite[page 44]{hps} and \cite{fe}), it follows that
\begin{itemize}
\item For  $i=2,3$, there are  bi-dimensional invariant manifolds    $W^{s}_{\gamma_i}$ (stable manifold) and
 $W^{u}_{\gamma_i}$ (unstable manifold), of class $C^{k-3}$, and $W^{u}_{\gamma_i}\cap W^{s}_{\gamma_i}=\gamma_i$.
\item For  $\gamma_1$, $W^{u}_{\gamma_i}=\emptyset$
\end{itemize}
Observe that $\Pi(W^{s}_{\gamma_1})=\Pi(W^{u}_{\gamma_i})=\mathcal{S}$, $i=2,3$. Let
$W_{1}=\Pi(W^{s}_{\gamma_2})\textrm{ and  } W_{2}=\Pi(W^{s}_{\gamma_3})$.
Then,
\begin{itemize}
\item there are two invariant manifolds
 (partially umbilic separatrix surfaces),
$W_{1}$ e $W_{2}$, both of class $C^{k-3}$,
\item there exists exactly one hyperbolic sector and one wedge  sector for
of the principal foliation $\mathcal{F}_{1}(\alpha)$.
\end{itemize}

See Fig. \ref{projD2}.

\begin{figure}[h]
\begin{center}
   \def\svgwidth{0.9\textwidth}
      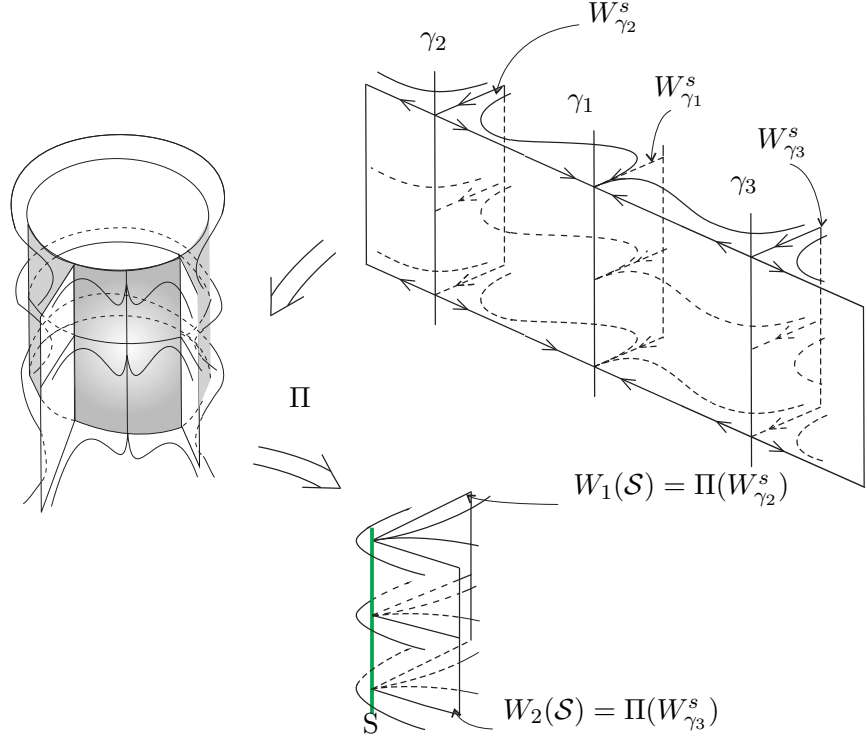
     \caption{Lie-Cartan resolution of $D_2$ with $\frac{a}{b}<2$.}
   \label{projD2}
    \end{center}
\end{figure}

Suppose that $\mathcal S$  is  of type $D_3$. By lemma \ref{lema1}, there are three curves of singularities
 $\gamma_1(u_3)$, $\gamma_2(u_3)$ and  $\gamma_3(u_3)$ for the Lie-Cartan vector field. By lemma \ref{lema2}, the curves $\gamma_1$, $\gamma_2$ and $\gamma_3$ are normally hyperbolic of saddle type. By Invariant Manifold Theory  (see \cite[page 44]{hps} and \cite{fe}), for  $i=1,2,3$, there are bi-dimensional invariant manifolds $W^{s}_{\gamma_i}$ and  $W^{u}_{\gamma_i}$, of class  $C^{k-3}$, with
 $W^{u}_{\gamma_i}\cap W^{s}_{\gamma_i}=\gamma_i$.
Moreover, $\Pi(W^{s}_{\gamma_1})=\Pi(W^{u}_{\gamma_i})=\mathcal{S},i=2,3$. Let
$W_{1}=\Pi(W^{u}_{\gamma_1}), W_{2}=\Pi(W^{s}_{\gamma_2})\textrm{ and }W_{3}=\Pi(W^{s}_{\gamma_3}),
$
Then,
\begin{itemize}
\item there are three invariant surfaces (partially umbilic separatrix surfaces)
$W_{1}$, $W_{2}$ e $W_{3}$,
all of class  $C^{k-3}$,
\item there are exactly three hyperbolic sectors of the principal foliation $\mathcal{F}_{1}(\alpha)$.
\end{itemize}

See Fig. \ref{projD3}.

\begin{figure}[h]
%
\begin{center}
   \def\svgwidth{0.9\textwidth}
    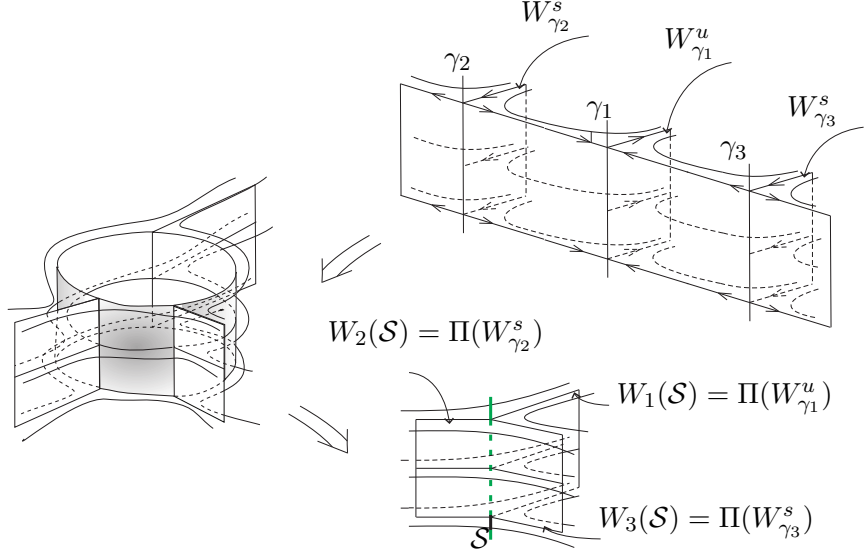
    \caption{Lie-Cartan resolution of a partially umbilic point of type  $D_3$.}
   \label{projD3}
    \end{center}
\end{figure}


\section{Principal foliations near a partially umbilic point  $D_{12}$}\label{D12}

\begin{lemma}\label{equilD12}
Let $\mathcal S$ be the partially umbilic set. If $p\in\mathcal S$ is of type $D_{12}$, then the Lie-Cartan vector field $X=X_{\mathcal L}$ has two lines of singularities $\zeta_1$ and $\zeta_2$.
\end{lemma}
\begin{proof}
The singular points of $X$ are given by
\begin{equation}\label{sistequlXcapD12}
\left\{
\begin{array}{l}
\left\{\begin{array}{c}
L_r(u_1,u_2,u_3)=0,\\
M_r(u_1,u_2,u_3)=0,
\end{array}\right.
\textrm{ (Partially Umbilic Points)}\\
\!\!\!A_3P^3+A_2P^2+A_1P+A_0=0\quad\textrm{ (see equation (\ref{CLCcap1}))}
\end{array}
\right.
\end{equation}
As $\det\left(\m\frac{\partial(L_r,M_r)}{\partial(u_1,u_2)}\Big|_{u_1=0=u_2=u_3}\right)=-b^2\neq0$ we can write $u_1=u_1(u_3)$ and $u_2=u_2(u_3)$
to solve
$L_r(u_1,u_2,u_3)=0$ and $M_r(u_1,u_2,u_3)=0$.
In the Monge chart,
 $\mathcal{S}$ is parametrized by

\begin{equation}\label{cpulemaD12}
\begin{array}{ll}
u_1=c_1(u_3)&={\m\frac {bq_{021}-q_{201}b-q_{111}c}{{b}^{2}}}u_3+O \left(2 \right),
\\
&\\
u_2=c_2(u_3)&=-\m\frac{q_{111}}{b}u_3+O\left(2 \right),
\end{array}
\end{equation}
Let
$
C(u_3,P)=A_3(c_1(u_3),c_2(u_3),u_3)P^3+A_2(c_1(u_3),c_2(u_3),u_3)P^2+$\\ $A_1(c_1(u_3),c_2(u_3),u_3)P+A_0(c_1(u_3),c_2(u_3),u_3)=X_4(c_1(u_3),c_2(u_3),u_3,P).
$
Direct calculation shows that
\begin{equation}\label{derXcpu1}
\frac{\partial C(u_3,P)}{\partial P}\Bigg|_{u_3=0,P=\frac{c}{b}}=\frac{c^2}{b}\neq0;\quad \frac{\partial C(u_3,P)}{\partial u_3}\Bigg|_{u_3=0,P=0}=\chi_{12}\neq0
\end{equation}
and
\begin{equation}\label{derXcpu2}
\quad \frac{\partial C(u_3,P)}{\partial P}\Bigg|_{u_3=0,P=0}=0,\quad\frac{\partial^2 C(u_3,P)}{\partial P^2}\Bigg|_{u_3=0,P=0}=-2c\neq0,
\end{equation}

So, by equations (\ref{derXcpu1}) and (\ref{derXcpu2}), there exists an
 unique curve of singularities $\zeta_1$, transversal
to the
 $P$ axis, passing through $u_3=0$, $P=c/b$; and there exists
 a unique curve, $\zeta_2$  tangent to the point $P$  axis   at the  origin.
\end{proof}

\begin{remark} \label{chi12}
The discriminant of $C(u_3,P)=0$ satisfies
\begin{equation}\label{discrdd}
D(u_3)=\chi_{12}u_3+O(2),
\end{equation}
where $\chi_{12}$ is given by equation \ref{chi2}.

\end{remark}

In order to obtain the configuration shown in Fig. \ref{D12S},
is sufficient to show that $\zeta_1$ is normally hyperbolic of saddle type and that
 $\zeta_2$ is of saddle-node type (non-hyperbolic)

\begin{lemma}\label{lemaRFD12} Let $X_{{\mathcal  L}}$  be the  Lie-Cartan vector field tangent to the Lie-Cartan hypersurface, and  $\zeta_1$, $\zeta_2$ the curves of singularities established in  lemma \ref{equilD12}. Then,
\begin{itemize}
\item[i)] $\zeta_1$ is a curve of singularities normally hyperbolic of saddle type;
\item[ii)] there exists a two dimensional center manifold containing $\zeta_2$, and the phase portrait of  $X_{{\mathcal L}}$ in a neighborhood of $\zeta_2$ is as shown in Fig. \ref{RetratoFaseSND12}.
\end{itemize}
\end{lemma}
\begin{figure}[h]
\psfrag{zeta2}{$\zeta_2$}
\psfrag{p}{P}
\begin{center}
 \def\svgwidth{0.8\textwidth}
    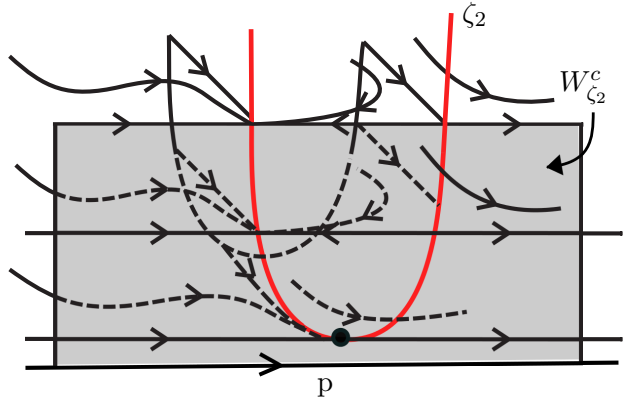
  \caption{Phase portrait of $X_{{\mathcal L}}$ in a neighborhood of $\zeta_2$ and the center manifold $W^c_{\zeta_2}$.  }\label{RetratoFaseSND12}
    \end{center}
    \end{figure}

\begin{proof}  For $u_3$ sufficiently small, the nonzero
eigenvalues of  $DX(\zeta_1(u_3))$ are:
$$
\lambda_1(u_3)=-\frac{b^2+c^2}{b}+O(1),\quad \lambda_2(u_3)=\frac{c^2}{b}+O(1).
$$

Therefore,   $\zeta_1$ is normally hyperbolic of saddle type in the neighborhood of $\zeta_1(0)$ (see \cite{hps} and \cite{fe}).

Direct calculation shows that
$
\frac{\partial\mathcal{L}}{\partial u_2}(0,0,0,0)=b\neq 0.
$
The implicit solution
 $u_2=u_2(u_1,u_3,P)$  of
 $\mathcal{L}(u_1,u_2,u_3,P)=0$, in a neighborhood of $0$,
is:
{\small
\begin{align*}
u_2(u_1,u_3,P)&=-{\frac {q_{111}}{b}}u_3-\frac{1}{2b}\left(B+2{\frac {q_{201
}q_{111}}{k-k_3}} \right)\!{u}_1^{2}+\left({\frac {q_{201}}{b}}-{\frac {q_{021}}{b}}+{\frac {cq_{111}}{{
b}^{2}}}\right)\!u_3P\\
&+\left({\frac {q_{201}q_{111}q_{021}}{ \left(k-k_3 \right) {b}^{2}}}
-{\frac {q_{201}q_{012}}{ \left(k-k_3 \right) b}}+{\frac {{q_{111}}^{3}}{
 \left(k-k_3 \right) {b}^{2}}}+{\frac {q_{111}C}{{b}^{2}}}-{\frac {q_{102}q_{111}}{
 \left(k-k_3 \right) b}}\right.\\
&\left.-{\frac {q_{111}{k}^{3}}{{b}^{2}}}-{\frac {Q_{211}}{b}}
\right)u_1u_3+\left(-{\frac {{q_{111}}^{3}q_{021}}{ \left(k-k_3 \right) {b}^{3}}}+
{\frac {q_{111}q_{021}q_{102}}{ \left(k-k_3 \right) {b}^{2}}} -\frac12{\frac {D{q_{111}}^{2}}{{b}^{3}}}\right.\\
&\left.+{\frac {{q_{111}}^{2}q_{012}}{ \left(k-k_3 \right) {b}^{2}}}
+{\frac {Q_{121} q_{111}}{{b}^{2}}}-\frac12{\frac {Q_{112}}{b}}
-{\frac {q_{012}q_{102}}{b \left(k-k_3 \right) }}
\right)u_3^2+O(3)
\end{align*}
}

Near
the origin, in the chart $(u_1,u_3,P)$ the vector field $X_{\mathcal{L}}$ defined in equation (\ref{CLCcap1}), with $a=2b$ and $u_2=u_2(u_1,u_3,P)$,
is given by:
\begin{equation}\label{campoLieCartanD12}
X_{\mathcal{L}}:=\left\{\begin{array}{lll}
\dot u_1&=&X_1(u_1,u_2(u_1,u_3,P),u_3,P)\\ 
\dot u_3&=&X_3(u_1,u_2(u_1,u_3,P),u_3,P)\\
\dot P&=&X_4(u_1,u_2(u_1,u_3,P),u_3,P)
\end{array}
\right.
\end{equation}
where
{\small
\begin{align*}
X_1&={\frac { \left(q_{021}b-cq_{111}-q_{201}b \right) }{b}}u_3+O(2)\\
X_3&=\Big(\mathcal{U}(u_1,u_2(u_1,u_3,P),u_3)+P\mathcal{V}(u_1,u_2(u_1,u_3,P),u_3)\Big)X_1\\
X_4&=\left(-B-3{\frac {q_{201}q_{111}}{k-k_3}} \right) u_1-bP+\\
&+\! \left(\!-Q_{211}\!+\!{\frac {q_{111}C}{b}}\!+\!{\frac {{k}^{3}}{b}}-\!{\frac {2q_{102}q_{111}}{k-k_3}}-\!{\frac {q_{201}q_{012}}{k-k_3}}\!+\!{\frac {2{u_3q_{111}}^{3}}{ \left(k-k_3 \right) b}}\!+\!{\frac {q_{201}q_{021}q_{111}}{ \left(k-k_3 \right) b}}\!\right)\!\!u_3-\\
&+O(2)
\end{align*}
}

Also $DX_{\mathcal{L}}(\zeta_2(P))$
has one zero eigenvalue and the other
ones  are given by:
$$
\lambda_1(P)=3bP^2-2cP+O(3)\textrm{ and  }\lambda_2(P)=-b+cP-2bP^2+O(3).
$$
It will be supposed that  $ b,\; c >0$ without loss of generality.

As $\lambda_2(P)<0$, for  $P$ sufficiently small,
by Invariant Manifold Theory  (see \cite[page 44]{hps} and \cite{fe}) there exists an invariant manifold $W^s(\zeta_2(P))$,
of  class
$C^{k-3}$ where $\zeta_2(P)$ is an attractor.

For $u_1=0$, there exists a two dimensional invariant center manifold, $W_{\zeta_2}^c$, of class $C^{k-3}$ that contains the curve of singularities  $\zeta_2$ in a neighborhood of $\zeta_2(0)$.

Below it will be shown that the phase portrait of $X$ restricted to $W^c_{\zeta_2}$ is as shown in Fig.
 \ref{D12SN}.

 \begin{figure}[h]
\begin{center}
 \def\svgwidth{0.6\textwidth}
    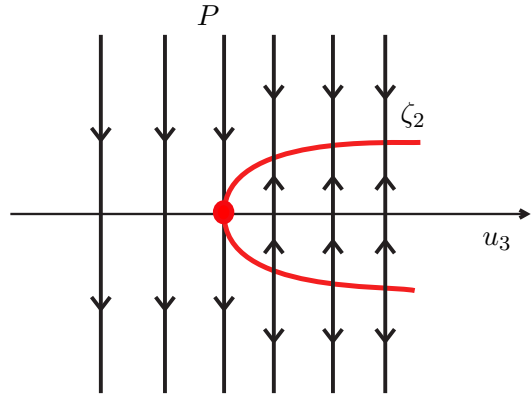
  \caption{Phase portrait of $X$ restricted to  the center manifold $W^c_{\zeta_2}$  near the point $\zeta_2(0)$.}
  \label{D12SN}
    \end{center}
\end{figure}

Perform the change of coordinates $(u_1,u_2,P)\rightarrow (u,w,\overline{P})$ such that the tangent plane to $W_{\zeta_2}^c$ at zero is the plane $w\overline{P}$.  Let

\begin{equation}\label{mcd12a}
\begin{array}{ll}
u_1&={\m\frac{\left(q_{021}b-q_{201}b-cq_{111}\right)}{b^2}} w+\m\frac{\chi_{12}}{b \left(Bk-Bk_3+3q_{201}q_{111} \right)} \overline{P}\\
u_3&=w\\
P&=-\m\frac { \chi_{12}}{{b}^{2}(k-k_3)}u+\m\frac { \chi_{12} }{{b}^{3}(k-k_3) }w-\m\frac {\chi_{12}}{{b}^{3}(k-k_3)}\overline{P}.
\end{array}
\end{equation}

Replacing \eqref{mcd12a} in  \eqref{campoLieCartanD12} it follows that:

\begin{equation*}
Y_r:=\left\{\begin{array}{lll}
\dot u&=&w+f(u,w,\overline{P})\\
\dot w&=&g(u,w,\overline{P})\\
\dot{\overline{P}}&=&-b\overline{P}+h(u,w,\overline{P})
\end{array}.
\right.
\end{equation*}

\noindent where $f(0,0,0)=g(0,0,0)=h(0,0,0)=0$,
 $\m\frac{\partial f}{\partial u}(0,0,0)
=\m\frac{\partial f}{\partial w}(0,0,0)=
\m\frac{\partial f}{\partial \overline{P}}(0,0,0)=0$,
$\m\frac{\partial g}{\partial u}(0,0,0)=\m\frac{\partial g}{\partial w}(0,0,0)
=\m\frac{\partial g}{\partial \overline{P}}(0,0,0)=0$, $\m\frac{\partial h}{\partial u}(0,0,0)=
\m\frac{\partial h}{\partial w}(0,0,0)=\m\frac{\partial h}{\partial \overline{P}}(0,0,0)=0$.

The center manifold $W^c_{\zeta_2}$, associated to $Y_r$, can be parametrized by
 $\overline{P}=\overline{P}(u,w)$. The restriction of $Y_r$ to $W^c_{\zeta_2}$ is given by:

\begin{equation} \label{camporvc}
Y_r\Big|_{W^c_{\zeta_2}}=\left\{\begin{array}{lll}
\dot u&=&U(u,w)\\
\dot w&=&W(u,w)
\end{array}
\right.,
\end{equation}
where $\frac{\partial U}{\partial w}(0,0)=1$ and
$
\frac{\partial^2 U}{\partial u^2}(0,0)=\displaystyle\frac{c\chi_{12}}{b^2(k-k_3)}\neq0.
$
 
So the phase portrait of  $Y_r$ restricted to the center manifold is given as in Fig. \ref{D12SN} (see \cite{dla}, Theorem $3.5$ pages 128 and 129 ).

Therefore, there exists invariant manifolds $
W^s_{\zeta_1(u_1)}$ and $W^c_{\zeta_2}$ of  $X_{\mathcal{L}}$ such that the phase portrait in the neighborhood  of these manifolds is as illustrated in Fig.  \ref{RetratoFaseSND12}.
\end{proof}

\subsection{End of proof of Theorem \ref{th:d12} }

Let $\mathcal{S}$ be a partially umbilic curve and $p\in\mathcal{S}$
be  of type $D_{12}$. By lemma \ref{equilD12}, the Lie-Cartan vector field has two curves of singularities   $\zeta_1$ and $\zeta_2$.
By lemma \ref{lemaRFD12}, $\zeta_1$ is normally hyperbolic of saddle type, and the phase portrait of the  Lie-Cartan vector field in a neighborhood of  $\zeta_2$ is as shown in Fig. \ref{RetratoFaseSND12}.
For  $\zeta_{1}$, there are  bi-dimensional invariant manifolds    $W^{s}_{\zeta_1}$ (stable manifold) and
 $W^{s}_{\zeta_1}$ (unstable manifold), of class $C^{k-3}$, and $W^{u}_{\zeta_1}\cap W^{s}_{\zeta_1}=\zeta_1$. For $\zeta_2$, there are  bi-dimensional invariant manifolds $W^{s}_{\zeta_2}$
  (stable manifold) and $W^{c}_{\zeta_2}$ (center manifold).

Observe that $\Pi(W^{u}_{\zeta_1})=\Pi(W^{c}_{\zeta_2})=\mathcal{S}$. Define $W=\Pi(W^{s}_{\zeta_1})$ and $W_3=\Pi(W^{s}_{\zeta_2})$.

This ends the proof of theorem \ref{th:d12} 
. For an illustration see Fig. \ref{projD12a}.

\begin{figure}[h]
\psfrag{d12}{$D_{12}$}
\psfrag{pi}{$\Pi$}
\psfrag{zeta1}{$\zeta_1$}
\psfrag{zeta2}{$\zeta_2$}
\psfrag{ws2}{$W^s(\zeta_2)$}
\begin{center}
 \def\svgwidth{0.8\textwidth}
    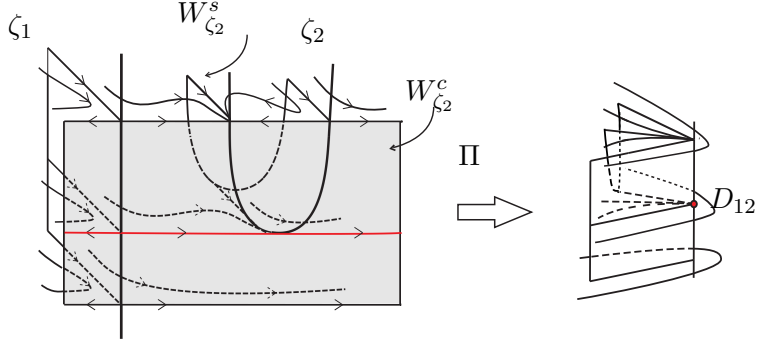
  \caption{ Normally hyperbolic singularity $\zeta_1$  and center manifold $W^c_{\zeta_2}$ containing
  $\zeta_2$.}
  \label{projD12a}
    \end{center}
\end{figure}

\section{Principal foliations near a partially umbilic point $D_{23}$}\label{sec:d23}

\subsection{ Local analysis of the Lie-Cartan vector field $X$}

\begin{lemma}\label{equilD23}
Let $\mathcal S$ be the partially umbilic set. If $p\in\mathcal S$ is of type $D_{23}$, then the Lie-Cartan vector field $X$ has three lines of singularities $\gamma_i$, $i=1,2,3$.
\end{lemma}

\begin{proof}
The singular points of $X$ are given by\\
\begin{equation}\label{sistequlXcap2}
\left\{
\begin{array}{l}
\left\{\begin{array}{c}
L_r(u_1,u_2,u_3)=0,\\
M_r(u_1,u_2,u_3)=0,
\end{array}\right.
\textrm{ (Partially Umbilic Points)}\\
\!\!\!A_3P^3+A_2P^2+A_1P+A_0=0\quad\textrm{ (see equation (\ref{CLCcap1}))}
\end{array}
\right.
\end{equation}
 
By lemma \ref{lema1},
$\mathcal S$ is a regular curve.
As $\det\left(\m\frac{\partial(L_r,M_r)}{\partial(u_2,u_3)}\Big|_{u_1=0=u_2=u_3}\right)=b(-q_{201}+q_{021})+cq_{111}\neq0$ we can write $u_2=u_2(u_1)$ and $u_3=u_3(u_1)$ in $L_r(u_1,u_2,u_3)=0$ and $M_r(u_1,u_2,u_3)=0$.
In the Monge chart (restricted), $\mathcal{S}$ is parametrized by
\begin{equation}\label{coordCurvaParcUmb}
\begin{array}{cl}
u_2=c_2(u_1)&\!\!\!\!=-\m\frac{1}{2(\left(q_{021}-q_{201} \right) b-cq_{111})} \left[q_{111}A+q_{111}C- \left(q_{021}-q_{201} \right)B\right.\\
&\\
&\!\!\!\!\left.-2\left(q_{201}q_{021}-q_{111}^{2} \right) q_{111}(k-k_3)^{-1}-2q_{111}{k}^{3}\right]u_1^{2}+O \left(3 \right) \\
&\\
u_3=c_3(u_1)&\!\!\!\!=\m\frac{1}{ \left(-2q_{201}b-2cq_{111}+42b \right)}\left[ bA-bC+cB-2{k}^{3}b\right.\\
&\\
&\!\!\!\!\left.+(k-k_3)^{-1}(-2\left(q_{111}^{2}+q_{201}^{2} \right) b+2q_{201}q_{111}c) \right] u_1^{2}+O\left(3 \right)
\end{array}
\end{equation}

From equation (\ref{CLCcap1}) with $a=b$, the second equation of
 (\ref{sistequlXcap2})  is given by:

\begin{equation}\label{equilibriosCub}
\left(b+O(1)\right)P^3+\left(-c+O(1)\right)P^2+\left(-b+O(1)\right)P+O(1)=0,
\end{equation}
The discriminant of the equation (\ref{equilibriosCub}) is given by:

\begin{equation}\label{discrdd1}
D(u_1)=-{\frac {1}{108}}{\frac {{c}^{2}+4{b}^{2}}{{b}^{2}}}+O(1)<0.
\end{equation}
Therefore, for $u_1$ sufficiently small the equation (\ref{equilibriosCub})
has three solutions $P_i(u_1)$($i=1,2,3$) given by
\begin{equation}\label{funcoesPiD23}
\begin{array}{ll}
P_1(u_1)&=0+O(1); P_2(u_1)={\m\frac {c+\sqrt {{c}^{2}+4{b}^{2}}}{2b}}+O(1);\\
&\\
P_3(u_1)&={\m\frac {c-\sqrt {{c}^{2}+4{b}^{2}}}{2b}}+O(1).
\end{array}
\end{equation}

So $X$ has three curves of singularities
 $\gamma_1$, $\gamma_2$ and  $\gamma_3$. Moreover $\Pi(\gamma_i)=\mathcal S$, $i=1,2,3$.

In the space
with coordinates
$(u_1,u_2,u_3,P)$ the curves of singularities of $X$ are given by:
 \begin{equation}\label{CPUsD23}
\gamma_i:\quad u_2=c_2(u_1),u_3=c_3(u_1)\; \textrm{ and } \;P=P_i(u_1),\quad i=1,2,3
\end{equation}
where $c_2$, $c_3$ and  $P_i$ $(i=1,2,3)$ are given by equations (\ref{coordCurvaParcUmb}) and
 (\ref{funcoesPiD23}),
 respec\-ti\-ve\-ly.
\end{proof}

\begin{lemma}\label{equiliD23} Let $X_{\mathcal{L}}$ be the  Lie-Cartan vector field
tangent to the    Lie-Cartan hypersurface.
The curves of singularities $\gamma_i$, $i=2,3$ are normally hyperbolic of saddle type of
  $X_{\mathcal{L}}$. Near the curve $\gamma_1$, the phase portrait of $X_{\mathcal{L}}$
is as shown in Fig.   \ref{D23SN1}.
\end{lemma}

\begin{figure}[h]
\begin{center}
 \def\svgwidth{0.6\textwidth}
    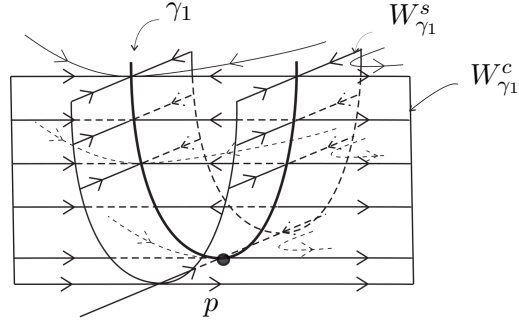
    \caption{Phase portrait of  $X_{\mathcal{L}}$
in a neighborhood of  $\gamma_1$.}
  \label{D23SN1}
    \end{center}
\end{figure}
\begin{proof}
The linearisation  $DX(\gamma_i(u_1))$, $i=2,3$, of
the
 Lie-Cartan vector field,
has two zero eigenvalues and the non-zero
ones  are given by:
$$
\lambda_1(u_1)=-b(P_i(0)^2+1)+O(1),\quad \lambda_2(u_1)=b(P_i(0)^2+1)+O(1),\quad i=2,3.
$$

Therefore for $u_1$ sufficiently small it follows that  $\lambda_1(u_1)<0$ and
$\lambda_2(u_1)>0,$  (assuming that $b>0$).

 Then   the curves $\gamma_2$ and  $\gamma_3$ are normally hyperbolic saddles of $X_{\mathcal L}$.

In a neighborhood of $P=P_i(0)$, $i=2,3$, the Lie-Cartan
 hypersurface is regular. In fact,

\begin{equation*}\label{LCGP2P3}
\frac{\partial\mathcal{L}}{\partial u_3}(0,0,0,P_i(0))=-
\m\frac{(b(q_{021}-q_{201})-cq_{111})(c+(-1)^i\sqrt{c^2+4b^2})}{2b^2}\neq0.
\end{equation*}

Next the analysis of  $X$  near
$\gamma_1(u_1)$, will be developed.

Direct calculation shows that $\frac{\partial\mathcal{L}}{\partial u_2}(0,0,0,0)=b\neq0$. The solution, $u_2=u_2(u_1,u_3,P)$ of the implicit equation $\mathcal{L}(u_1,u_2(u_1,u_3,P),u_3,P)=0$ near the $(0,0,0,0)$ is given by
 
{\small
\begin{align*}
u_2(u_1,u_3,P)&=-{\frac {q_{111}}{b}}u_3-\frac{1}{2b}\left(B+2{\frac {q_{201
}q_{111}}{k-k_3}} \right)\!{u}_1^{2}+\left({\frac {q_{201}}{b}}-{\frac {q_{021}}{b}}+{\frac {cq_{111}}{{
b}^{2}}}\right)\!u_3P\\
&+\left({\frac {q_{201}q_{111}q_{021}}{ \left(k-k_3 \right) {b}^{2}}}-
{\frac {q_{201}q_{012}}{ \left(k-k_3 \right) b}}+
{\frac {{q_{111}}^{3}}{ \left(k-k_3 \right) {b}^{2}}}+
{\frac {q_{111}C}{{b}^{2}}}-{\frac {q_{102}q_{111}}{ \left(k-k_3 \right) b}}\right.\\
&\left.-{\frac {q_{111}{k}^{3}}{{b}^{2}}}-{\frac {Q_{211}}{b}}
\right)u_1u_3+\left(-{\frac {{q_{111}}^{3}q_{021}}{ \left(k-k_3 \right) {b}^{3}}}
+{\frac {q_{111}q_{021}q_{102}}{ \left(k-k_3 \right) {b}^{2}}} -\frac12{\frac {D{q_{111}}^{2}}{{b}^{3}}}\right.\\
&\left.+{\frac {{q_{111}}^{2}q_{012}}{ \left(k-k_3 \right) {b}^{2}}}
+{\frac {Q_{121} q_{111}}{{b}^{2}}}-
\frac12{\frac {Q_{112}}{b}}-{\frac {q_{012}q_{102}}{b \left(k-k_3 \right) }}
\right)u_3^2+O(3)
\end{align*}
}
Near the origin, in the chart $(u_1,u_3,P)$, the vector field $X_{\mathcal L}$ is given by
\begin{equation}\label{campoLieCartan2}
X_{\mathcal{L}}:=\left\{\begin{array}{lll}
\dot u_1&=&X_1(u_1,u_2(u_1,u_3,P),u_3,P)\\  
\dot u_3&=&X_3(u_1,u_2(u_1,u_3,P),u_3,P)\\
\dot p&=&X_4(u_1,u_2(u_1,u_3,P),u_3,P)
\end{array},
\right.
\end{equation}
where,

\begin{align*}
X_1&={\frac { \left(q_{021}b-cq_{111}-q_{201}b \right) }{b}}u_3+O(2)\\
X_3&=\Big(\mathcal{U}(u_1,u_2(u_1,u_3,P),u_3)+P\mathcal{V}(u_1,u_2(u_1,u_3,P),u_3)\Big)X_1\\
X_4&=\left(-B-3{\frac {q_{201}q_{111}}{k-k_3}} \right) u_1-bP+\\
&+\! \left(\!-Q_{211}\!+\!{\frac {q_{111}C}{b}}\!+\!{\frac {{k}^{3}}{b}}-
\!{\frac {2q_{102}q_{111}}{k-k_3}}-\!{\frac {q_{201}q_{012}}{k-k_3}}\!+
\!{\frac {2{q_{111}}^{3}}{ \left(k-k_3 \right) b}}\!+\!
{\frac {q_{201}q_{021}q_{111}}{ \left(k-k_3 \right) b}}\!\right)\!\!u_3-\\
&+O(2)
\end{align*}

The eigenvalues of $DX_{\mathcal{L}}(\gamma_1(u_1))$ are given by:

\begin{align*}
\lambda_2(u_1)=-b+O(1);\lambda_3(u_1)={\frac {q_{201}\left( \left(q_{021}-q_{201} \right)
 b-cq_{111} \right) }{b \left(k-k_3 \right) }}u_1+O(2).
\end{align*}

As $\lambda_2(u_1)<0$, for $u_1$ sufficiently small, it follows by
Invariant Manifold Theory, (see \cite[page 44]{hps} and \cite{fe})
that there exists an  invariant manifold $W^s_{\gamma_1}$ (stable manifold),
of class  $C^{k-3}$, see Fig.  \ref{D23SN1}.

For $u_1=0$, there exists a two dimensional invariant center manifold, $W_{\gamma_1}^c$, of class $C^{k-3}$ that contains the curve of singularities  $\gamma_1$ in a neighborhood of $\gamma_1(0)$.

Below it will be shown that the phase portrait of $X$ restricted to $W^c_{\gamma_1}$ is as shown in Fig. \ref{D23VC}.

Consider  $X_{\mathcal{L}}$ restricted to  $W_{\gamma_1}^c$ and perform the change of coordinates such that
the tangent plane to $W_{\gamma_1}^c$ at zero is the plane $uw$.  Let,

\begin{equation}\label{mcl}
\begin{array}{ll}
u_1&={\m\frac{\left(q_{021}b-q_{201}b-cq_{111}\right)}{b}} u\\
&+\frac{1}{b \left(Bk-Bk_3+3q_{201}q_{111} \right)}\left(-2q_{111}bq_{102}-Q_{211}kb+bk_3Q_{211}-bq_{012}q_{201}+\right.\\
&\left.+2q_{111}^{3}+q_{111}kC-q_{111}Ck_3-q_{111}{k}^{4}+q_{111}{k}^{3}k_3+q_{201}q_{111}q_{021} \right) w\\
u_3&=w\\
P&=-\m\frac { \left(q_{021}b-q_{201}b-cq_{111} \right)\left(Bk-Bk_3+3q_{201}q_{111} \right)}{{b}^{2}(k-k_3)}u\\
&+\m\frac { \left(q_{021}b-q_{201}b-cq_{111} \right)  \left(Bk-Bk_3+3q_{201}q_{111} \right) }{{b}^{3}(k-k_3) }w\\
&-\m\frac { \left(q_{021}b-q_{201}b-cq_{111} \right)  \left(Bk-Bk_3+3q_{201}q_{111}
\right)}{{b}^{3}(k-k_3)}\overline{P}.
\end{array}
\end{equation}

From equation (\ref{mcl}) it follows that  $X_{\mathcal{L}}$ is given by:

\begin{equation*}
Y_r:=\left\{\begin{array}{lll}
\dot u&=&w+f(u,w,\overline{P})\\
\dot w&=&g(u,w,\overline{P})\\
\dot{\overline{P}}&=&-b\overline{P}+h(u,w,\overline{P})
\end{array}.
\right.
\end{equation*}

\noindent where $f(0,0,0)=g(0,0,0)=h(0,0,0)=0$,
 $\m\frac{\partial f}{\partial u}(0,0,0)
=\m\frac{\partial f}{\partial w}(0,0,0)=
\m\frac{\partial f}{\partial \overline{p}}(0,0,0)=0$,
$\m\frac{\partial g}{\partial u}(0,0,0)=\m\frac{\partial g}{\partial w}(0,0,0)
=\m\frac{\partial g}{\partial \overline{p}}(0,0,0)=0$, $\m\frac{\partial h}{\partial u}(0,0,0)=
\m\frac{\partial h}{\partial w}(0,0,0)=\m\frac{\partial h}{\partial \overline{p}}(0,0,0)=0$.

The center manifold $W^c$, associated to $Y_r$ can be parametrized by
 $\overline{P}=\overline{P}(u,w)$. The restriction of $Y_r$ to $W^c$ is given by:

\begin{equation} \label{camporvcd12}
Y_r\Big|_{W^c}=\left\{\begin{array}{lll}
\dot u&=&U(u,w)\\
\dot w&=&W(u,w)
\end{array}
\right.,
\end{equation}
where

 $$\aligned U(u,w)=& 
-\frac{1}{2b(k-k_3)} [ \left(k-{   k_3} \right)  \left(-2\,b{k}^{3}+{  A}\,b+c{  B}-
 {  C}\,b \right) \\+&\left(-2\,{{   q_{111}}}^{2}+2\,{{   q_{201}}}^{2}
  \right) b+2\,{   q_{111}}\,{   q_{201}}\,c]u^2+\frac {b({  q_{021}-q_{201}})-c {  q_{111}} }{b}w\\
  +&  A_{11}uw+ A_{02} w^2+O(3)\\
  W(u,w)=&\left[ \frac {b({  q_{021}-q_{201}})-c {  q_{111}} }{b^2(k-k_3)}\right
   ] \left[q_{201}\,u w  +(b  q_{102}-  q^{2}_{111}) {w}^{2} \right]+O(3)
 \endaligned$$
 
 Performing the change of variables $u_1=u, w_1=U(u,w)$ it is obtained the vector field
 $$\aligned u_1^{\prime} =& w_1\\
  w_1^{\prime}=& w_1[ \frac {c {  q_{111}}-b({  q_{021}-q_{201}}) }{b^4(k-k_3)}  \chi_{23}u_1+  A_w w_1+O(2)]\endaligned $$
 
 The coefficients $A_{11}$, $A_{02}$ and $A_w$   have long expressions in terms of the other coefficients but are not relevant to determine the phase portrait in the center manifold.

So the phase portrait of  $Y_r$ restricted to the center manifold is given as in Fig. \ref{D23VC} (see \cite{dla}, Theorem $3.5$, pages 128 and 129).
  
\begin{figure}[h]
\psfrag{beta1}{$\gamma_1$}
\begin{center}
\def\svgwidth{0.5\textwidth}
    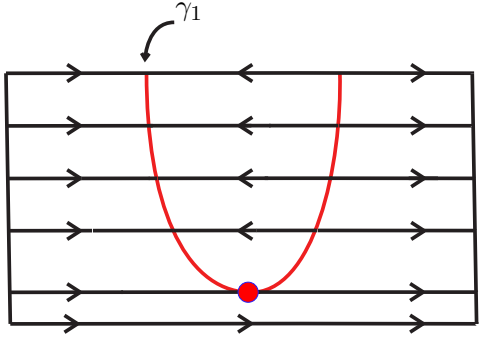
    \caption{Phase portrait of $Y_r\Big|_{W^c}$.}
  \label{D23VC}
    \end{center}
\end{figure}

Therefore, there exist  invariant manifolds
 $
W^s_{\gamma_1(u_1)}$ and $W^c_{\gamma_1(0)}$ of  $X_{\mathcal{L}}$ and the phase portrait in the neighborhood  of these manifolds is as illustrated in Fig. \ref{D23SN1}.
\end{proof}

\subsection{End of proof of Theorem \ref{th:d23} }\label{sub:d23}

Let $\mathcal{S}$ be a partially umbilic curve and $p\in\mathcal{S}$ of type $D_{23}$. By lemma \ref{equilD23}, there exists three curves of singularities $\gamma_1(u_1)$, $\gamma_2(u_1)$ and $\gamma_3(u_1)$ of the Lie-Cartan vector field. By lemma \ref{equiliD23}, the curves  $\gamma_2$ and $\gamma_3$ are normally hyperbolic for the vector field $X_{\mathcal L}$.
Near $\gamma_1$ the phase portrait of $X_{\mathcal{L}}$ is as in Fig.  \ref{D23SN1}.

For  $\gamma_1$, there are  bi-dimensional invariant manifolds $W^{s}_{\gamma_1}$ (stable manifold) and $W^{c}_{\gamma_1}$ (center manifold).

For  $i=2,3$, there are  bi-dimensional invariant manifolds    $W^{s}_{\gamma_i}$ (stable manifold) and
 $W^{u}_{\gamma_i}$ (unstable manifold), of class $C^{k-3}$, and $W^{u}_{\gamma_i}\cap W^{s}_{\gamma_i}=\gamma_i$.

Observe that $\Pi(W^{u}_{\gamma_2})=\Pi(W^{u}_{\gamma_3})=\Pi(W^{s}_{\gamma_1})=\mathcal{S}$, and
define $W=\Pi(W^{c}_{\gamma_1})$, $W_1=W_{1}(\mathcal{S})=\Pi(W^{s}_{\gamma_2})$ and  $W_2=W_{2}(\mathcal{S})=\Pi(W^{s}_{\gamma_3})$.

A  connected component of $  W_1\setminus {\mathcal S}$
 is invariant by \fp 1
and the other is invariant by \fp 2. The same conclusion  for $W_2$.

The invariant surfaces  $W_1$ and $W_2$ are  regular and there exist four   leaves $F_1$ and $F_4$ in $W_1$ and $F_2$ and $F_3$ in $W_2$   asymptotic to the partially umbilic $D_{23}$.

It follows that $F_1$ and $F_2$ are leaves of \fp 1 and $F_3$ and $F_4$ are leaves of \fp 2.  See Figs.  \ref{D23S} and \ref{projD23}. %
 
  This ends the proof of Theorem \ref{th:d23}.
 
\begin{figure}[h]
\begin{center}
 \def\svgwidth{0.9\textwidth}
    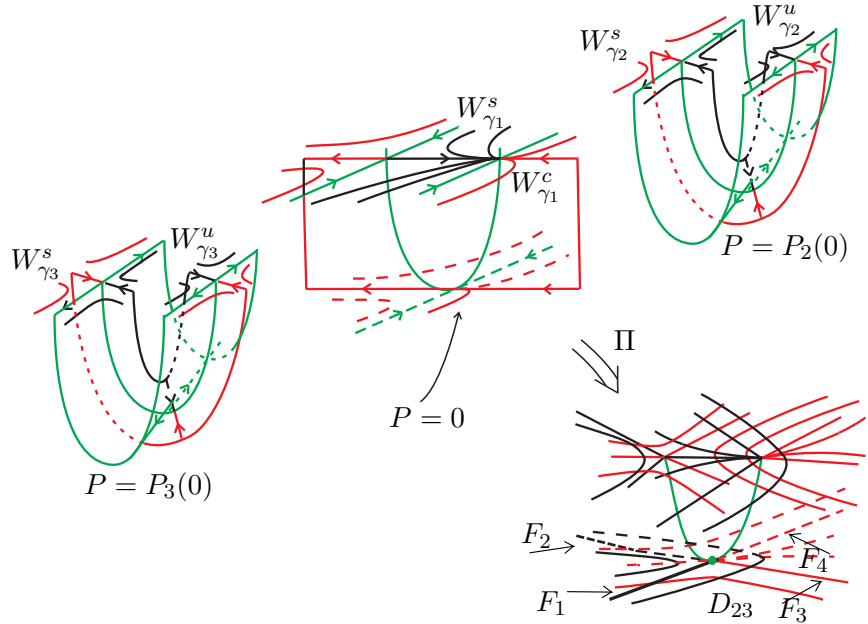
\caption{Projections of the integral curves of $X_{\mathcal{L}}$.
  \label{projD23}}
    \end{center}
\end{figure}

\section{A Summarizing Theorem}     \label{sec:summa}

Consider the space $\mathbb{J}^r(\M^3,\R^4)$ of $r$-jets of immersions $\alpha$  of $\M^3$ into $\mathbb R^4$, endowed with the structure of Principal Fiber Bundle. The base is $\M^3\times{\mathbb R}^4$, the fiber is the space $\mathcal{J}^r(3,4)$, where $\mathcal{J}^r(3,4)$ is the space of $r$-jets of immersions of ${\mathbb R}^3$ to ${\mathbb R}^4$, preserving the respective origins. The structure group, $\mathcal{A}_{+}^r$, is the product of the group $\mathcal{L}_+^r(3,3)$ of $r$-jets of
 diffeomorphisms of ${\mathbb R}^3$  preserving origin and the orientation, 
 acting on the right by coordinate changes, and the group $\mathcal{O}_+(4,4)$ of positive isometries; the action on the left consists of a positive rotation of ${\mathbb R}^4$.

Each 
 $4$-jet
 of an immersion at a partially umbilic point is of the form $(p,\tilde{p},w)$ with $(p,\tilde{p})\in\M^3\times{\mathbb R}^4$ 
and $w$  is in the orbit of a polynomial immersion $(u_1,u_2,u_3,h(u_1,u_2,u_3))$, where
\begin{equation}\label{funcaohJatos}
\begin{array}{ll}
h(u_1,u_2,u_3)\!\!\!\!&=\frac{k}{2}(u_1^2+u_2^2)+\frac{k_3}{2}u_3^2+\frac a6u_1^3+\frac b2u_1u_2^2+\frac c6u_2^3
+\frac{q_{003}}{6}u_3^3+\\
&\\
&+\frac{q_{012}}{2}u_2u_3^2+q_{111}u_1u_2u_3+\frac{q_{021}}{2}u_2^2u_3+\frac{q_{102}}{2}u_1u_3^2+\frac{q_{201}}{2}u_1^2u_3+\\
&\\
&+\frac{A}{24}{u}_1^{4}+\frac{B}{6}{u}_1^{3}u_2+\frac{C}{4}{u}_1^{2}{u}_2^{2}+\frac{D}{6}u_1u_2^{3}+\frac{E}{24}u_2^{4}+\frac{Q_{004}}{24}u_3^{4}+\\
&\\
&+\frac{Q_{013}}{6}u_3^{3}u_2+\frac{Q_{103}}{6}u_3^{3}u_1+\frac{Q_{022}}{4}u_2^{2}u_3^{2}+\frac{Q_{202}}{4}u_1^{2}u_3^{2}+\frac{Q_{031}}{6}u_3u_2^{3}+\\
&\\
&+\frac{Q_{112}}{2}u_3^{2}u_1u_2+\frac{Q_{301}}{6}u_1^{3}u_3+\frac{Q_{121}}{2}u_3u_2^{2}u_1+\frac{Q_{211}}{2}u_3u_2u_1^{2}\\
\end{array}
\end{equation}
The general quadratic part of $\tilde{h}$, where $(u_1,u_2,u_3,\tilde{h})$ is in the orbit of $h$, has the form
$$
k_{110}u_1u_2+k_{101}u_1u_3+k_{011}u_2u_3+\frac{k_{200}}{2}u_1^2+\frac{k_{020}}{2}u_2^2+\frac{k_{002}}{2}u_3^2.
$$
The manifold of partially umbilic jets, $(\mathcal{PU})^4$, is defined by  
the 
condition that the symmetric matrix
$$
\left(\begin{array}{ccc}
k_{200}&k_{110}&k_{101}\\
k_{110}&k_{020}&k_{011}\\
k_{101}&k_{011}&k_{002}
\end{array}\right)
$$
has two 
 equal
eigenvalues. In \cite{ga1}, R. Garcia showed that $(\mathcal{PU})^4$ is a submanifold of codimension $2$ in $\mathcal{J}^4(3,4)$.
Sse also Lax \cite{lax}.

The manifold of umbilic jets, $\mathcal{U}^4$, is defined by condition $k_{200}=k_{020}=k_{002}$ and $k_{110}=k_{101}=k_{011}=0$. It is a closed submanifold of codimension $5$ in $\mathcal{J}^4(3,4)$.

\begin{remark}
The expression in equation (\ref{funcaohJatos}) is a representative of the orbit of partially umbilic jets where the term $\m\frac{d}{2}u_1^2u_2$ 
vanishes, as can be obtained  by a rotation.
\end{remark}

Define below a canonic stratification of $\mathcal{J}^4(3,4)$. The term canonic means that the strata are invariant under the action of the group $\mathcal{A}^4_+=\mathcal{O}_+(4,4)\times\mathcal{L}^4_+(3,3)$.  It is to the orbits of this actions that reference is made in the following definition.

\begin{definition} \label{def:can-strat}
\begin{enumerate}
\item \textit{The Umbilic Jets}: $\mathcal{U}^4$ are those in the orbits of $\mathcal{J}^4(u_1,u_2,u_3,h)$, where $h=h(u_1,u_2,u_3)$, is as in equation (\ref{funcaohJatos}). It is a closed submanifold of codimension $5$ in $\mathcal{J}^4(3,4)$.
\item \textit{ The Partially Umbilic Jets}: $(\mathcal{PU})^4$ are those in the orbits of $\mathcal{J}^4(u_1,u_2,u_3,h)$, where $h=h(u_1,u_2,u_3)$, is as in equation (\ref{funcaohJatos}). It is a submanifold of codimension $2$ in $\mathcal{J}^4(3,4)$.
\item \textit{The Darbouxian Partially Umbilic Jets}: $\mathcal{D}^4=(\mathcal{D}_1)^4\cup(\mathcal{D}_2)^4\cup(\mathcal{D}_3)^4$ are  those in the orbits of $\mathcal{J}^4(u_1,u_2,u_3,h)$, where $h=h(u_1,u_2,u_3)$ is as in equation (\ref{funcaohJatos}), where
\begin{enumerate}
\item $(\mathcal{D}_1)^4$: $\left(\m\frac{c}{2b}\right)^2-\m\frac{a}{b}+2<0$
\item $(\mathcal{D}_2)^4$: $\left(\m\frac{c}{2b}\right)^2+2>\m\frac{a}{b}>1$, $a\neq2b$
\item $(\mathcal{D}_3)^4$: $\m\frac{a}{b}<1.$
\end{enumerate}
\item \textit{The Non-Darbouxian Partially Umbilic Jets}: $\mathcal{ND}^4$ are  those in the orbits of $\mathcal{J}^4(u_1,u_2,u_3,h)$, where $h=h(u_1,u_2,u_3)$ is as in equation (\ref{funcaohJatos}) whose expression satisfy $$b(b-a)=0\textrm{ or }$$ $$b(b-a)\neq0 \textrm{ and } a-2b=0\textrm{ or }c^2-ab(a-2b)=0.$$
This manifold can be further partitioned into
\begin{enumerate}
\item $(\mathcal{D}_{12})^4$: Defined by the orbits of jets with\\
$
\left\{\begin{array}{c}
a=2b\ne 0, c\ne 0, \textrm{ and }\chi_{12}\neq0, \textrm{ or}\\
c^2-4b(a-2b)=0\textrm{ and }\chi_{12}^{*}\neq0
\end{array}
\right.
$
\noindent See equation \eqref{chi2}.
 
\item $(\mathcal{D}_{23})^4$: Defined by the orbits of jets with $a=b\neq0$ and $\chi_{23}\neq0$. See definition \ref{defD23}.
\end{enumerate}
\end{enumerate}

\end{definition}

The canonic stratification of $\mathcal{J}^4(3,4)$ induces a canonic stratification of $\mathbb{J}^4(\M^3,{\mathbb R}^4)$ whose strata are principal sub-bundles with codimension equal to that of their fibers, which are the canonic strata of $\mathcal{J}^4(3,4)$, as defined above in items $1$, $2$, $3$ and $4$.

The collection of sub-bundles which stratify $\mathbb{J}^4(\M^3,{\mathbb R}^4)$ will be called Partially Umbilic Stratification. The strata are: $\mathbb{PU}^4(\M^3,{\mathbb R}^4)$ corresponding to $(\mathcal{PU})^4$;  $(\mathbb{D}_i)^4(\M^3,{\mathbb R}^4)$, $i=1,2,3$, corresponding to the strata of  Darbouxian partially umbilic
jets $(\mathcal{D}_i)^4$, $i=1,2,3$, and so on, one bundle for each of the strata in definition \ref{def:can-strat}.

\begin{theorem}\label{th:summat}
Let $\M^3$ be a compact, oriented, smooth manifold.
In the space of immersions $Imm^r(\M^3, {\mathbb R}^4)$,  
endowed with the $C^4-$topology, 
the following properties define open and dense sets.
 
\begin{itemize}
\item[i)] The set ${\mathcal S}(\alpha)$  of partially umbilic points of $\alpha\in Imm^r(\M^3, {\mathbb R}^4)$  is
empty or it is a smooth submanifold of codimension 2, a curve,  of   $\M^3$ stratified as follows:
 \item[ii)]   The points $D_1$, $D_2$ e $D_3$ are distributed along a finite number of arcs along a connected component $\mathcal S$.
\item[iii)]   The points $D_{12}$ and  $D_{23}$ are isolated along a connected component of $\mathcal S$ and are the common border points of the arc in the previous item.
 
\end{itemize}
\end{theorem}

\begin{proof} The density follows from Thom transversality theorem in the space of jets. See \cite{levine} and \cite{thom}. 
The openness in the  compact case  of  $\M^3$  is clear from the analysis in the proofs of the three main theorems of this paper.
 First treat  the isolated points  $D_{12}$ and  $D_{23}$  with statements and proofs  in sections  \ref{sec:d12_PU},  \ref{D12},  and  in sections \ref{sec:d23_PU} and \ref{sec:d23}, respectively.
 In fact the definitions of these points and their isolated   character depend  on the local openness of  
  transversality, in its simplest implicit function theorem form,  applied to mappings depending only up to fourth derivatives of the immersions, (\cite{thom}),   and on the local openness  of    hyperbolicity and normal hyperbolicity and transversal saddle - nodes   of equilibria of  vector fields  which, again,  depend on the fourth order derivatives of the immersion. See  \cite{fe}, \cite{hps} and  \cite{mp}.
   This analysis proofs  that each such point  $p_i$  has a neighboring closed arc  $I_i$,  where the conclusions of the theorems \ref{th:d12} and \ref{th:d23} hold for an open set of immersions. 
  Then apply the same idea to each of the finite  compact arcs  $J_{ij}$  of Darbouxian arcs,  which are complementary to the interior of   arcs $I_i$.
  
  The  openness  on the immersion invoked above hold also  along the arcs   $J_{ij}$. For this case the mappings  and vector fields involved in the arguments depend only
  on the third order derivatives of the immersion.  
  \end{proof}

\begin{remark}
In the case of non-compact theorem  \ref{th:summat}  has an obvious
 analogous replacing { \emph  open and dense} by {\em  residual} and the standard 
$C^r -$ topology by the   {\em ph Whitney topology}. See \cite{thom}.
\end{remark}

\section{Concluding Remarks} \label{sec:CR}

In this work it is proved that the singularities of the principal line fields of a generic  immersion
of a compact oriented three-dimensional manifold $M^3$ into $\mathbb {R}  ^4 $  
consist  of a regular curve $\mathcal S$ of {\it
partially umbilic points}, at which only two of the principal
curvatures coincide.

The curve $\mathcal S$ consists of arcs of
{\it transversal partially umbilic points,}
at which
it crosses
the partially umbilic planes.

These arcs have as common extremes a discrete set of
semi-Darbouxian $D_{23}$ points, at which
the contact of $\mathcal S$ with the umbilic plane is quadratic.
These points are the common extremes, or points of transition,
between  partially
umbilic Darbouxian  $D_2$ and  $D_3$ arcs.

The transversal arcs contain a discrete set of semi-Darbouxian
 $D_{12}$ points. 
 These
points are the common extremes, or points of transition,  between
the  partially umbilic  Darbouxian  $D_1$ and  $D_2$ arcs.

The
transversal structure   of principal curvature lines, along the
Darbouxian arcs $D_1,\; D_2$  and $D_3$  is   reminiscent of the
Darbouxian umbilics  in surfaces of $\mathbb {R}^3$ as described by Darboux \cite{da}
and Gutierrez - Sotomayor \cite{gs1}. At the transition points
$D_{12}$ and $D_{23}$, it 
 closely resembles  the bifurcations of
umbilic points appearing
in generic one parameter families of surfaces $\mathbb {R}^3$ studied  in the works of
Gutierrez, Garcia and Sotomayor \cite{gs4}.  

It must be emphasized, however, that  the three dimensional analysis carried out in this 
paper is  unavoidable.  In fact, only on the highly non-generic 
case  of 
Frobenius integrability of 
 the  plane distribution  spanned by the 
principal line fields $\mathcal L_1$  and $\mathcal L_2$  around the 
partially umbilic curve $\mathcal S_{12}$  the structure of  the 
principal foliations $\mathcal F_1$ and $\mathcal F_2$ can be considered 
as those appearing in one-parameter families of surfaces. 
See remark \ref{rem:nonfrob}.
The same  holds for the partially umbilic curve  ${\mathcal S}_{12}$ and   the structure of  the 
principal foliations ${\mathcal F}_2$ and ${\mathcal F}_3$.

The partition of the partially umbilic curve into the arcs  $D_1,\; D_2$ and $ D_3$ and
 the transition points $D_{23}$ and  $ D_{12} $,  together with the  stratified
structure of the partially umbilic separatrix surfaces,  consisting
of  all the principal lines  approaching $\mathcal S$, established in this
work, constitute a  natural extension to hypersurfaces in $\re^4$
of the results  of Darboux  for umbilic points for surfaces in  $\re^3$, \cite{da},
as reformulated  and extended by  Gutierrez and Sotomayor in   \cite{gs1} and Gutierrez, Garcia and Sotomayor in  \cite{gs4}. 

The stratified structure of the invariant  separatrix  manifolds foliated 
by   curvature lines approaching the partially umbilic curves,
 established in 
 Theorems \ref{th:d123}, \ref{th:d12} and \ref{th:d23},
  being 
 an \emph{invariant}  of each  
  one dimensional foliation  with singularities $\mathcal F _i$, $(i=1,2,3)$,
 under \emph{principal equivalence},
 is an important ingredient  added 
 to the original study
 of 
  Garcia in \cite{ga1} and \cite{ga3}.  
 
  It is a crucial  structural step in order to 
  formulate   an  improved     genericity theorem 
stating that an immersion of  $M^3$ into $\mathbb R^4$ has each
one of its three principal foliations having all its separatrix strata intersecting pairwise transversally.  
 By a {\em principal equivalence }  is  understood  a  homeomorphism of $M^3$ which preserves the principal foliations with singularities  $\mathcal F_i$, $(i=1,2,3)$,  one at the time.     
 However, the non-integrability result established in  remark  \ref{rem:nonfrob}
   makes  unfeasible  the simultaneous  {\em   principal structural stability } considerations for pairs of
  principal foliations as was done in the case of surfaces. 
  See \cite{gs1}  and \cite{gs6}. 
  Nevertheless, on the individual basis,  additional elaboration of  
  the normal hyperbolicity methods 
    used  in this work 
  leads to the  local principal structural stability  
   of each of the  foliations 
   at the partial umbilic  curves consisting of
   $D_1$,  $D_2$, $D_3 $, $D_{12}$ and $D_{23}$ points. 

\section{Appendix: Coordinate expressions for geometric functions appearing in this work} \label{sec:App}

Consider 
a  chart  $(u_1,u_2,u_3)$ and an isometry $R$  as in section \ref{sec:Monge_jet} so that 
 the immersion $\alpha$  composed with  $R$  has the Monge form: 

 $(u_1,u_2,u_3) \to (u_1, u_2, u_3, h \uu )$,  where $h$ is given by equation \eqref{eq:di}.
 In this section will be obtained the  coordinate   expressions  of functions that are essential 
 for the calculations carried out in this work.  

Direct calculation 
with 
\eqref{eq:di} 
 gives the expressions below:
{\small
\begin{align*}
\frac{\partial h}{\partial u_1}&=k u_1 +\frac12au_1^{2}+\frac12 bu_2^{2}+\frac12q_{102}u_3^{2}+q_{201}u_1u_3+q_{111}u_2u_3+\frac16Q_{103}u_3^{3}\\
&+\frac12 Bu_1^{2}u_2+\frac16 Du_2^{3}+\frac12 Q_{202}u_1u_3^{2}+\frac12 Q_{301}u_1^{2}u_3+\frac12Q_{112}u_3^{2}u_2\\
&+Q_{211}u_3u_2u_1+\frac12Q_{121}u_3u_2^{2}+\frac16Au_1^{3}+\frac12Cu_1u_2^{2}+O(4),
\end{align*}

\begin{align*}
\frac{\partial h}{\partial u_2}&=k u_2 +\frac12cu_2^{2}+q_{021}u_2u_3+bu_1u_2+q_{111}u_1u_3+\frac12Du_1u_2^{2}+\frac12q_{012}u_3^{2} \\
&+\frac16Bu_1^{3}+\frac12Q_{112}u_3^{2}u_1+\frac16Q_{211}u_3u_1^{2}+Q_{121}u_3u_2u_1+\frac12Q_{022}u_2u_3^{2}\\
&+\frac16Eu_2^{3}+\frac16Q_{013}u_3^{3}+\frac12Q_{031}u_3u_2^{2}+\frac12Cu_1^{2}u_2+O(4),
\end{align*}

\begin{align*}
\frac{\partial h}{\partial u_3}&=k_3 u_3+\frac12q_{021}u_2^{2}+\frac12q_{201}u_1^{2}+q_{102}u_1u_3+q_{012}u_3u_2+q_{111}u_1u_2+\frac12q_{003}u_3^{2}\\
&+\frac16Q_{301}u_1^{3}+Q_{112}u_3u_1u_2+\frac12Q_{013}u_3^{2}u_2+\frac16Q_{031}u_2^{3}+\frac12Q_{202}u_1^{2}u_3\\ &+\frac12Q_{211}u_2u_1^{2}+\frac12Q_{121}u_2^{2}u_1+\frac12Q_{022}u_2^{2}u_3+\frac16Q_{004}u_3^{3}+\frac12Q_{103}u_3^{2}u_1\\
&+O(4),
\end{align*}
}

\subsection{First Fundamental Form} \label{ssec:AfirstFF}

The coefficients of the first fundamental form $(g_{ij})$ of $\alpha$ in the  Monge chart  $(u_1,u_2,u_3)$ 
are:

\begin{align*}
g_{11}&=\left<\alpha_{u_1},\alpha_{u_1}\right>=1+{k}^{2}u_1^{2}+kau_1^{3}+bku_2^{2}u_1+kq_{102}u_3^{2}u_1+2kq_{111}u_2u_3u_1\\
&+2kq_{201}u_1^{2}u_3+h.o.t.,
\end{align*}
\begin{align*}
g_{12}&=\left<\alpha_{u_1},\alpha_{u_2}\right>={k}^{2}u_1u_2+kq_{111}u_3u_1^{2}+\frac {kc}{2}u_1u_2^{2}+ \left(kq_{021}+q_{201}k \right) u_3u_2u_1\\
&+ k\left(b+\frac a2 \right) u_1^{2}u_2+\frac{kq_{012}}{2}u_3^{2}u_1+\frac{bk}{2}u_2^{3}+kq_{111}u_3u_2^{2}+q_{102}u_3^{2}ku_2
+h.o.t.,
\end{align*}
{\small
\begin{align*}
g_{13}&=\left<\alpha_{u_1},\alpha_{u_3}\right>=k_3u_1u_3+\frac{kq_{201}}{2}u_1^{3}+kq_{111}u_1^{2}u_2+ \left(kq_{102
}+\frac {ak_3}{2} \right) u_3u_1^{2}\\
&+\frac{kq_{021}}{2}u_2^{2}u_1+kq_{012}u_3u_2u_1+ \left(\frac {kq_{003}}{2}+q_{201} \right)u_3^{2}u_1+\frac {bk_3}{2}u_3u_2^{2} +q_{111}k_3u_3^{2}u_2\\
&+\frac{q_{102}k_3}{2}u_3^{3}+h.o.t.,
\end{align*}
}
\begin{align*}
g_{22}&=\left<\alpha_{u_2},\alpha_{u_2}\right>=1+{k}^{2}u_2^{2}+kcu_2^{3}+2kq_{021}u_2^{2}u_3+kq_{012}u_3^{2}u_2+2kbu_1u_2^{2}\\
&+2kq_{111}u_2u_3u_1+h.o.t.,
\end{align*}

\begin{align*}
g_{23}&=\left<\alpha_{u_2},\alpha_{u_3}\right>=kk_3u_2u_3+\frac12kq_{201}u_2u_1^{2}+kq_{111}u_2^{2}u_1+ \left(kq_{102}+bk_3 \right)u_3u_2u_1\\
&+q_{111}k_3u_3^{2}u_1+\frac12kq_{021}u_2^{3}+\left(kq_{012}+\frac{ck_3}{2} \right) u_3u_2^{2}+\left(\frac12kq_{003}+q_{021}k_3\right)u_3^{2}u_2\\
&+\frac{k_3q_{012}}{2}u_3^{3}
+h.o.t.,
\end{align*}
\begin{align*}
g_{33}&=\left<\alpha_{u_3},\alpha_{u_3}\right>=1+k_3^2u_3^{2}+q_{021}k_3u_2^{2}u_3+q_{003}k_3u_3^{3}+2q_{102}k_3u_1u_3^{2}+q_{201}k_3u_1^{2}u_3\\
&+2k_3q_{012}u_3^{2}u_2+2k_3q_{111}u_1u_2u_3+h.o.t.,
\end{align*}

\subsection{Normal Vector} \label{ssec:Anormal}

Taylor  expansion of the components  of the  unit normal of $\alpha$,
$
N={\mathcal N}_\alpha=\left(\alpha_{u_1} \wedge \alpha_{u_2}  \wedge  \alpha_{u_3} \right)/{ | \alpha_{u_1} \wedge \alpha_{u_2}  \wedge  \alpha_{u_3} |},
$
 gives  the following expressions for $N=(n_1,n_2,n_3,n_4)$ in a neighborhood of $(0,0,0)$:

{\small
\begin{align*}
n_1&=-u_1k-\frac12q_{102}u_3^{2}-\frac12au_1^{2}-q_{111}u_2u_3 -\frac12bu_2^{2}-q_{201}u_1u_3-\frac12Q_{112}u_3^{2}u_2\\
&-Q_{211}u_1u_2u_3-\frac16Q_{103}u_3^{3}-\frac12Q_{121}u_3u_2^{2}-\frac16Du_2^{3}+\left(-\frac12C+\frac12k^3\right)u_1u_2^{2}\\
&+\left(-\frac12Q_{202}+\frac12kk_3^2\right)u_1u_3^{2}-\frac12Q_{301}u_1^{2}u_3+\left(-\frac16A+\frac12k^3\right)u_1^{3}-B{u_1}^{2}u_2\\
&+h.o.t,
\end{align*}
\begin{align*}
n_2&=-u_2k-\frac12q_{012}u_3^{2}-bu_1u_2-\frac12cu_2^{2}-q_{111}u_1u_3-q_{021}u_2u_3-Q_{121}u_3u_2u_1\\
&-\frac16Q_{013}u_3^{3}-Q_{031}u_3u_2^{2}-\frac12Q_{112}u_3^{2}u_1-\frac12Q_{211}u_3u_1^{2}-\frac12Du_1u_2^{2}\\
&+\left(-\frac12C+\frac12k^3\right)u_1^{2}u_2+\left(-\frac12Q_{022}+\frac12kk_3^2\right)u_2u_3^{2}+\left(-\frac16E+\frac12k^3\right)u_2^{3}\\
&-\frac16Bu_1^{3}+h.o.t.,
\end{align*}
\begin{align*}
n_3&=-k_3u_3-\frac12q_{201}u_1^{2}-\frac12q_{003}u_3^{2}-q_{102}u_1u_3-q_{111}u_1u_2-\frac12q_{021}u_2^{2}-q_{012}u_3u_2\\
&-\frac12Q_{211}u_2u_1^{2}-\frac12Q_{103}u_3^{2}u_1+\left(-Q_{202}+\frac12k^2\right)u_1^{2}u_3-\frac12Q_{013}u_3^{2}u_2\\
&-\frac12Q_{121}u_2^{2}u_1+\left(-\frac12Q_{022}+\frac12k^2\right)u_2^{2}u_3-\frac16Q_{301}u_1^{3}-Q_{112}u_3u_1u_2\\
&-\frac16Q_{031}u_2^{3}+\left(-\frac16Q_{004}+\frac12\right)u_3^{3}+h.o.t.,
\end{align*}
\begin{align*}
n_4&=1-\frac12u_2^{2}{k}^{2}-\frac12u_1^{2}{k}^{2}-\frac{k_3^2}{2}u_3^{2}-\frac12kau_1^{3}+\left(-\frac12k-1 \right)q_{012} u_3^{2}u_2\\
&+ \left(-\frac12k-1 \right) q_{102}u_1u_3^{2}-\frac12q_{003}u_3^{3}-\frac32bku_2^{2}u_1+ \left(-2k-1\right)q_{111} u_1u_3u_2\\
&-\frac12kcu_2^{3}+ \left(-k-\frac12 \right) q_{201}u_3u_1^{2}+ \left(-k-\frac12 \right)q_{021} u_3u_2^{2}+h.o.t..
\end{align*}
}

\subsection{Second  Fundamental Form} \label{ssec:AsecondFF}

The coefficients of the second fundamental form $(\lambda_{ij})$
are:
{\small

\begin{align*}
\lambda_{11}&=\left<\frac{\partial^2\alpha}{\partial u_1^2},N\right>=k+au_1+q_{201}u_3+ \left(-\frac12{k}^{3}+\frac12A\right) u_1^{2}+Bu_1u_2+Q_{301}u_1u_3\\
&+ \left(-\frac12{k}^{3}+\frac12C \right) u_2^{2}+Q_{211}u_3u_2+ \left(-\frac{kk_3^2}{2}+\frac12Q_{202} \right)u_3^{2}+h.o.t.,
\end{align*}

\begin{align*}
\lambda_{12}&=\left<\frac{\partial^2\alpha}{\partial u_1\partial u_2},N\right>=q_{111}u_3+bu_2+Q_{121}u_3u_2+\frac12Q_{112}u_3^{2}+Q_{211}u_3u_1+\frac12Du_2^{2}\\ &+Cu_1u_2+\frac12Bu_1^{2}+h.o.t.,
\end{align*}

\begin{align*}
\lambda_{13}&=\left<\frac{\partial^2\alpha}{\partial u_1\partial u_3},N\right>=q_{201}u_1+q_{111}u_2+q_{102}u_3+\frac12Q_{103}u_3^{2}+Q_{112}u_3u_2 +Q_{211}u_2u_1\\ &+Q_{202}u_1u_3+\frac12Q_{301}u_1^{2}+\frac12Q_{121}u_2^{2}+h.o.t.,
\end{align*}
\begin{align*}
\lambda_{22}&=\left<\frac{\partial^2\alpha}{\partial u_2^2},N\right>=k+cu_2+q_{021}u_3+bu_1+ \left(-\frac12{k}^{3}+\frac12C\right){u_1}^{2}+Q_{121}u_3u_1\\
&+Du_1u_2+\left(-\frac12{k}^{3}+\frac12E\right)u_2^{2}+Q_{031}u_3u_2+\left(-\frac{k_3^2k}{2}+\frac12Q_{022} \right) {u_3}^{2}+O(3),
\end{align*}
\begin{align*}
\lambda_{23}&=\left<\frac{\partial^2\alpha}{\partial u_2\partial u_3},N\right>=q_{012}u_3+q_{111}u_1+q_{021}u_2+\frac12Q_{013}u_3^{2}+Q_{112}u_3u_1+Q_{121}u_2u_1\\ &+\frac12Q_{031}u_2^{2}+Q_{022}u_2u_3+\frac12Q_{211}u_1^{2}+O(3),
\end{align*}
\begin{align*}
\lambda_{33}&=\left<\frac{\partial^2\alpha}{\partial u_3^2},N\right>=k_3+q_{012}u_2+q_{102}u_1+q_{003}u_3+\left(-\frac1{k_3{k}^{2}}{2}+\frac12Q_{202}\right)u_1^{2}\\
&+Q_{112}u_1u_2+Q_{013}u_3u_2+\left(-\frac{k_3{k}^{2}}{2}+\frac12Q_{022} \right) u_2^{2}+ \left(\frac12Q_{004}-\frac{k_3^3}{2}\right) u_3^{2}+O(3),
\end{align*}
}

\subsection{The principal curvature $k_3$ } \label{ssec:Ak3}

The principal curvature  $k_3$, which is smooth near the origin, is given by
 
{\small
\begin{align*}
k_3(u_1,u_2,u_3)&= k_3  + q_{102}u_1+q_{012}u_2+q_{003}u_3\\
&-\frac{k_3k^2(k-k_3)-kQ_{202}+2q_{201}^2+k_3Q_{202}+2q_{111}^2}{k-k_3}u_1^2\\
&+\frac {kQ_{112}-2q_{201}q_{111}-k_3Q_{112}-2q_{111}q_{021}}{k-k_3}u_1u_2\\
&-\frac {k_3{k}^{2}(k-k_3)-kQ_{022}+2{q_{111}}^{2}+2{q_{021
}}^{2}+k_3Q_{022}}{k-k_3}u_2^2\\
&+\frac {kQ_{103}-2q_{012}q_{111}-2q_{201}q_{102
}-Q_{103}k_3}{k-k_3}u_1u_3\\
&+\frac {kQ_{013}-2q_{102}q_{111}-2q_{021}q_{012
}-k_3Q_{013}}{k-k_3}u_2u_3\\
&+\frac {kQ_{004}-3k_3^3k+3k_3^4-2{q_{102}}^{2}-k_3Q_{004}-2{q_{012}}^{2}}{k-k_3}u_3^2+h.o.t.
\end{align*}

}

\subsection{The principal direction $e_3$ } \label{ssec:Ae3}

The smooth principal direction $e_3(q)=(du_1,du_2,du_3)$ is defined by
 
\begin{align*}
\frac{du_1}{du_3}&=\frac{U_1(u_1,u_2,u_3)}{W_1(u_1,u_2,u_3)}, \;\;\;\;
\frac{du_2}{du_3} =\frac{V_1(u_1,u_2,u_3)}{W_1(u_1,u_2,u_3)}.
\end{align*}

The functions $ U_1, V_1$ and $ W_1$ are given by solving the linear system (\ref{eq:1}),  taking $i=3$.

Using the equation \eqref{eq:UVW1} and the subsections \ref{ssec:AfirstFF}, \ref{ssec:AsecondFF} and \ref{ssec:Ak3} it is obtained:

{\small
\begin{equation}\label{funcaoU1}
\begin{array}{ll}
{U}_1&\!\!\!\!=
(k_{3}-k ) q_{201}u_1+ (k_{3}-k )
q_{111}u_2+ (k_3-k ) q_{102}u_3\\
&+ (\frac 12{Q_{301}} (k_3-k ) -q_{201} (b+q_{102} )  ) u_1^{2}\\ 
&+ (q_{021}b-cq_{111}+\frac 12 ({k_3} -k ) Q_{121}+q_{111}q_{012} )u_2^{2}\\
&+ (-q_{201}c+q_{201}q_{012}+q_{111}q_{102}+Q_{211} (k_3 -k)  ) u_1u_2\\
&+ [ q_{111}^{2}+ (k_3^{2}k-Q_{202})(k-k_3 )+q_{201}(q_{003}-q_{021})+(q_{102}-b)q_{102}] u_1u_3\\
&+ (q_{012}b-q_{102}c+ (-k+k_3 )Q_{112}+q_{003}q_{111}+q_{012}q_{102} )u_2u_3\\
&+ (q_{111}q_{012}-q_{102}q_{021}+ \frac 12({k_3}-k ) Q_{103}+q_{003}q_{102} )u_3^{2}+h.o.t.
\end{array}
\end{equation}
}

{\small
\begin{equation}\label{funcaoV1}
\begin{array}{ll}
{V}_1&\!\!\!\!=
(k_3-k ) q_{111}u_1+ (k_3-k )q_{021}u_2+ (k_3-k ) q_{012}u_3\\
&+(q_{111} (q_{102}-a ) -\frac 12 Q_{211} (k-k_3)) u_1^{2} + (q_{111}b+q_{012}q_{021}+  \frac 12({k_3}-k )Q_{031} ) {u_2}^{2}\\
&+ [  q_{111}q_{012}+q_{201}b+(k_3-k)Q_{121}+q_{102}q_{021}-aq_{021} ] u_1u_2\\
&+ (q_{003}q_{111}+q_{012} (q_{102}-a ) + (k_3-k )Q_{112} ) u_1u_3\\
&+ [ q_{102}b+q_{111}^{2}-q_{201}q_{021}+q_{021}q_{003}+q_{012}^{2}+(k_3^{2}k-Q_{022}) (k-k_3 )] u_2u_3\\
&+ [\frac 12 (k_3-k)Q_{013}-q_{201}q_{012}+q_{111}q_{102}+q_{003}q_{012}] u_3^{2}+h.o.t.
\end{array}
\end{equation}
}
{\small
\begin{equation}\label{funcaoW1}
\begin{array}{ll}
{W}_1&\!\!\!\!= (k-k_3)^2+(k-k_3)(a+b-2q_{102})u_1+(k-k_3)(c-2q_{012})u_2\\
&+(k_3-k)[2q_{003}-q_{021}-q_{201}]u_3+h.o.t.\\
\end{array}
\end{equation}
}

\subsection{The field of planes $\mathcal{P}_3$ } \label{ssec:P3}

The plane field ${\mathcal P}_3$ defined by $\omega=0$ (see equation \ref{eqcampoplanos})
can be written as $du_3={\mathcal U} du_1+{\mathcal V}du_2$ where

$$\mathcal{U}=-\frac{[g_{11}U_1+ g_{12}{V_1}+g_{13} {W_1}]}{[g_{13}{U_1}+  g_{23}{V_1}+g_{33}{W_1}]}\;\;\; \text{and}\;\;\; \; \mathcal{V}=-\frac{[g_{12}{U_1} +g_{22}{V_1}+
g_{23}{W_1}]}{[g_{13}{U_1}+  g_{23}{V_1}+g_{33}{W_1}]} $$

 The Taylor expansions
  of  $\mathcal U$ and $\mathcal V$ 
  in a neighborhood of zero, 
  are
  given by:
{\small
\begin{equation}\label{funcaoU}
\begin{array}{ll}
\mathcal{U}\!\!\!\!&=\!\m\frac{1}{2(k-k_3)^2}[2q_{201}(k-k_3)u_1+2q_{111}(k-k_3)u_2+2q_{102}(k-k_3)u_3\\
&\\
&+(-2q_{021}b+2q_{111}q_{012}+kQ_{121}-k_3Q_{121})u_2^{2}+\left(2q_{003}q_{102}-2q_{201}q_{102}\right.\\
&\\
&\left.+(k-k_3)Q_{103}-2q_{111}q_{012}\right) u_3^{2}+(2kk_3^{3}-2q_{201}^{2}-2k_3Q_{202}-2k_3^2{k}^{2}\\
&\\
&+2kQ_{202}+2q_{102}^{2}-2q_{111}^{2}-2aq_{102}+2q_{201}q_{003})u_3u_1\\
&\\
&+\left(-k_3Q_{301}+2q_{201}q_{102}-2q_{201}a+Q_{301}k \right) u_1^{2}+(-2k_3Q_{112}+2q_{012}q_{102}\\
&\\
&-2q_{111}q_{021}+2q_{003}q_{111}-2q_{012}b-2q_{201}q_{111}+2kQ_{112})u_3u_2+(-2aq_{111}\\
&\\
&+2q_{102}q_{111}+2q_{201}q_{012}-2q_{111}b-2k_3Q_{211}+2Q_{211}k)u_1u_2]+h.o.t.
\end{array}
\end{equation}
}

{\small
\begin{equation}\label{funcaoV}
\begin{array}{ll}
\mathcal{V}\!\!\!\!&=\!\m\frac{1}{2(k-k_3)^2}[2q_{021}(k-k_3)u_2+2q_{111}(k-k_3)u_1+2q_{012}(k-k_3)u_3\\
&\\
&+(kQ_{031}-2q_{111}b-2q_{021}c-k_3Q_{031}+2q_{012}q_{021}) u_2^{2}+(-2k_3^2{k}^{2}\\
&\\
&-2q_{111}^{2}+2k_3^3k-2q_{012}+2kQ_{022}-2k_3Q_{022}-2q_{021}^{2}-2q_{102}b+2q_{012}^{2}\\
&\\
&+2q_{021}q_{003}) u_3u_2+(-2q_{111}q_{201}+2q_{003}q_{111}+2kQ_{112}-2q_{111}q_{021}\\
&\\
&+2q_{012}q_{102}-2q_{012}b-2k_3Q_{112}) u_3u_1+ (-k_3Q_{211}-2q_{111}b+2q_{102}q_{111}\\
&\\
&+Q_{211}k) u_1^{2}+(kQ_{013}+2q_{003}q_{012}-k_3Q_{013}-2q_{012}q_{021}-2q_{102}q_{111})u_3^{2} \\
&\\
&(-2q_{111}c-2q_{201}b+2kQ_{121}+2q_{102}q_{021}-2q_{021}b-2k_3Q_{121}\\
&\\
&+2q_{111}q_{012}) u_1u_2]+h.o.t.
\end{array}
\end{equation}
}
\subsection{The first fundamental form restricted to the plane field  $\mathcal{P}_3$ } \label{ssec:firstP3}

The first fundamental form $I=\sum g_{ij}du_idu_j$ restricted to the plane field  $\mathcal{P}_3$ is given by:

\begin{equation*}
I_r(du_1,du_2)=I_{\alpha}\Big|_{du_3=\mathcal{U}du_1+\mathcal{V}du_2}=E_rdu_1^2+2F_rdu_1du_2+G_rdu_2^2,\\
\end{equation*}
{\small
\begin{equation}\label{Er}
\begin{array}{ll}
E_{r}&=1+\left({k}^{2}+\m\frac{ {q_{201}}^{2}}{(k-k_3)^2}\right){u_1}^{2}+2\m\frac{q_{111}q_{201}}{(k-k_3)^2}u_1u_2
+2\m\frac {q_{102}q_{201}}{(k-k_3)^2}u_1u_3\\
&\\
&+\m\frac{{q_{111}}^{2}}{(k-k_3)^2}{u_2}^{2}+2\m\frac {q_{102}q_{111}}{(k-k_3)^2}u_2u_3
+\m\frac {{q_{102}}^{2}}{(k-k_3)^2}{u_3}^{2}+O(3);
\end{array}
\end{equation}

\begin{equation}\label{Fr}
\begin{array}{ll}
F_{r}&= \m\frac{q_{102}q_{111}+q_{201}q_{012}}{(k-k_3)^2}u_3u_1+\m\frac{q_{102}q_{012}}{(k-k_3)^2}{u_3}^{2}+\m\frac{q_{111}q_{021}}{(k-k_3)^2}{u_2}^{2}\\
&\\
&+\left(\m\frac{{q_{111}}^{2}}{(k-k_3)^2}+{k}^{2}+\m\frac{q_{201}q_{021}}{(k-k_3)^2}\right)u_1u_2
+\m\frac{q_{102}q_{021}+q_{111}q_{012}}{(k-k_3)^2}u_2u_3\\
&\\
&+\m\frac{q_{201}q_{111}}{(k-k_3)^2}{u_1}^{2}+O(3);
\end{array}
\end{equation}

\begin{equation}\label{Gr}
\begin{array}{ll}
G_{r}&=1+\m\frac{{q_{012}}^{2}}{(k-k_3)^2}{u_3}^{2}+\left({k}^{2}+\m\frac{{q_{021}}^{2}}{(k-k_3)^2}\right){u_2}^{2}+\m\frac{{q_{111}}^{2}}{(k-k_3)^2}{u_1}^{2}\\
&\\
&+\m\frac{2q_{021}q_{012}}{(k-k_3)^2}u_2u_3+\m\frac{2q_{111}q_{021}}{(k-k_3)^2}u_1u_2+\m\frac{2q_{012}q_{111}}{(k-k_3)^2}u_1u_3+O(3);
\end{array}
\end{equation}
}
\subsection{The second fundamental form
restricted to the plane field $\mathcal{P}_3$ } \label{ssec:secondP3}

The second fundamental form 
$II=\sum\lambda_{ij}du_idu_j$ restricted to the plane field  $\mathcal{P}_3$   is given by:
\begin{equation*}
II_r(du_1,du_2)=II_{\alpha}\Big|_{du_3=\mathcal{U}du_1+\mathcal{V}du_2}=e_rdu_1^2+2f_rdu_1du_2+g_rdu_2^2,
\end{equation*}
 
{\small
\begin{equation}\label{er}
\begin{array}{ll}
e_{r}&=k+au_1+q_{201}u_3+ \left(\m\frac{A}{2}+\m\frac{\left(2k-k_3 \right){q_{201}}^{2}}{(k-k_3)^2}-\m\frac{{k}^{3}}{2} \right) {u_1}^{2}\\
&\\
&+ \left(\m\frac { \left(2k-k_3 \right) {q_{111}}^{2}}{ \left(k-k_3 \right) ^{2}}+\m\frac C2-\m\frac{k ^{3}}{2}\right){u_2}^{2}+ \left(\m\frac{Q_{202}}{2}+\m\frac{\left(2k-k_3\right){q_{102}}^{2}}{ \left(k-k_3 \right)^{2}}-\m\frac{kk_3^2}{2} \right){u_3}^{2}\\
&\\
&+ \left(B+{\m\frac { \left(4k-2k_3 \right) q_{201}q_{111}}{ \left(k-k_3 \right) ^{2}}} \right) u_1u_2+ \left({\m\frac { \left(4k-2k_3 \right) q_{102}q_{201}}{ \left(k-k_3 \right) ^{2}}}+Q_{301} \right) u_1u_3\\
&\\
&+ \left(Q_{211}+{\m\frac{ \left(4k-2k_3 \right) q_{102}q_{111}}{\left(k-k_3 \right) ^{2}}} \right) u_2u_3+O(3);
\end{array}
\end{equation}
}

\begin{equation}\label{fr}
\begin{array}{ll}
f_{r}&=bu_2+ q_{111}u_3+\left(\m\frac{B}{2}+{\m\frac { \left(2k-k_3 \right)q_{201}q_{111}}{ \left(k-k_3 \right) ^{2}}} \right)u_1^{2}\\
&\\
&+ \left({\m\frac { \left(2k-k_3 \right)q_{021}q_{111}}{ \left(k-k_3 \right) ^{2}}}+\frac{D}{2}\right)u_2^{2}+ \left({\m\frac { \left(2k-k_3 \right) q_{012}q_{102}}{ \left(k-k_3 \right) ^{2}}}+\frac{Q_{112}}{2} \right) u_3^{2}\\
&\\
&+ \left(C+{\m\frac { \left(2k-k_3 \right)q_{021}q_{201}}{ \left(k-k_3 \right) ^{2}}}+{\m\frac { \left(2k-k_3 \right) q_{111}^{2}}{ \left(k-k_3 \right) ^{2}}} \right) u_1u_2\\
&\\
&+ \left(Q_{211}+{\m\frac { \left(2k-k_3 \right) q_{012}q_{201}}{ \left(k-k_3 \right) ^{2}}}+{\m\frac { \left(2k-k_3\right) q_{102}q_{111}}{ \left(k-k_3 \right) ^{2}}} \right) u_1u_3\\
&\\
&+ \left(2Q_{121}+{\m\frac { \left(2k-k_3 \right) q_{012}q_{111}}{ \left(k-k_3 \right) ^{2}}}+{\m\frac { \left(2k-k_3 \right) q_{021}q_{102}}{ \left(k-k_3 \right) ^{2}}} \right)u_2u_3+O(3);
\end{array}
\end{equation}

{\small
\begin{equation}\label{gr}
\begin{array}{ll}
g_{r}&=k+bu_1+cu_2+q_{021}u_3 + \left(-\m\frac{{k}^{3}}{2}+\m\frac{C}{2}+{\m\frac { \left(2k-k_3 \right) {q_{111}}^{2}}{ \left(k-k_3 \right) ^{2}}} \right) u_1^{2}\\
&\\
&+ \left(-\m\frac{{k}^{3}}{2}+\m\frac{E}{2}+{\m\frac{ \left(2k-k_3 \right) q_{021}^{2}}{ \left(k-k_3 \right) ^{2}}} \right) u_2^{2}+ \left(\m\frac{Q_{022}}{2}-\m\frac{kk_{3}^2}{2}+{\m\frac { \left(2k-k_3 \right) q_{012}^{2}}{ \left(k-k_3 \right) ^{2}}} \right) u_3^{2}\\
&\\
&+ \left(D+{\m\frac { \left(-2k_3+4k \right) q_{021}q_{111}}{ \left(k-k_3 \right) ^{2}}} \right) u_1u_2+ \left(Q_{121}+{\m\frac { \left(-2k_3+4k \right)q_{012}q_{111}}{ \left(k-k_3 \right) ^{2}}} \right) u_1u_3\\
&\\
&+ \left(Q_{031}+{\m\frac { \left(-2k_3+4k \right)q_{012}q_{021}}{ \left(k-k_3 \right) ^{2}}} \right)u_2u_3+O(3)

\end{array}
\end{equation}
}

\subsection{Coefficient $\chi_{12}^{*}$ in a Monge chart }\label{chi1s}

The two patterns for the  failure  of the  discriminant condition  $D_2$,   
$ \frac{a}{b}=\left(\frac{c}{2b}\right)^2+2$  and $ a=2b$,  keeping the transversality condition $b\neq a$,   are permuted by a rotation  in the $(u,v)$-plane preserving the form of  equation (\ref{eq:di}).  

When $ \frac{a}{b}=\left(\frac{c}{2b}\right)^2+2$, $b(b-a)\ne 0$, 
the characterization of the partially umbilic point  $D_{12}$
  is as appearing in terms of $\chi_{12}$  
  in remark \ref{chi12}, 
and   in terms of saddle node equilibrium in equation  
\ref{camporvc}.

When   $ \frac{a}{b}=\left(\frac{c}{2b}\right)^2+2$, an   analysis similar to that carried out in section 
\ref{D12}  gives equations  with  a coefficient  proportional to 
 $\chi_{12}^{*}$  given below, instead of $\chi_{12}$.

\begin{equation}\label{chi11}
\chi_{12}^{*}=\chi_{11}+\chi_{22}
\end{equation}
\noindent where,
{\small
$$\aligned \frac{\chi_{11}}{k-k_3}=&
 16\,{b}^{3}c  \left(-b{  q_{201}}+b{  q_{021}
}-{  q_{111}}\,c \right) { A}\\
-&4\,{b}^{2} \left(5\,{c}^{3}{  q_{111}
}-8\,{b}^{3}{  q_{201}}+8\,{b}^{3}{  q_{021}}-4\,{b}^{2}c{  q_{111}}-4\,b
{c}^{2}{  q_{021}}+4\,b{c}^{2}{  q_{201}} \right)  { B}\\
 +& 4\,b   \left(8\,{b}^{2}-2\,{c}^{2} \right)  \left({b}^{2}+{c}^{2} \right) 
{  q_{111}}-bc \left(8\,{b}^{2}-{c}^{2} \right) q_{021}+bc \left(8
\,{b}^{2}-{c}^{2} \right)  q_{201}
  { C}\\
+&c \left(-{c}^{4}{  q_{111}}+8\,{b}^{3}c{  q_{201}}+
32\,{b}^{4}{  q_{111}}-8\,{b}^{3}c{  q_{021}}+12\,{b}^{2}{c}^{2}{ 
 q_{111}} \right)   { D}\\
+&2\,b{c}^{2}{  q_{111}
}\, \left(4\,{b}^{2}+{c}^{2} \right)   { 
E}-bc \left(4\,{b}^{2}+{c}^{2} \right)  \left(8\,{b}^{2}-{c}^{2}
 \right)    { Q_{121}}\\
 -&4\,{b}^{2} \left(4\,{b
}^{2}+{c}^{2} \right)  \left(2\,{b}^{2}-{c}^{2} \right)   { Q_{211}}
-2\,{b}^{2}{c}^{2} \left(4\,{b}^{2}+{c}^{2}
 \right)  { Q_{031}}+4\,{b}^{3}c \left(4\,{
b}^{2}+{c}^{2} \right)    { Q_{301}}
\endaligned$$
}

{\small
$$\aligned \chi_{22}=&
4\,{b}^{2}c{k}^{3} \left(4\,{b}^{2}+{c}^{2} \right)  \left(k-{ k_3
} \right) {  q_{201}}-4\,{b}^{2}c{k}^{3} \left(4\,{b}^{2}+{c}^{2}
 \right)  \left(k-{ k_3} \right) {  q_{021}}\\
 -& 2\,b{k}^{3} \left(4b^2-c^2\right)  \left(4\,{b}^{2}+{c}^{2} \right) 
 \left(k-{ k_3} \right) {  q_{111}}-4\,{b}^{2} \left(4\,{b}^{2}+{c}
^{2} \right)  \left(2\,{b}^{2}-{c}^{2} \right) {  q_{201}}\,{  q_{012}}\\
-&6\,{b}^{2}{c}^{2} \left(4\,{b}^{2}+{c}^{2} \right) {  q_{012}}\,{ 
 q_{021}}-bc \left(4\,{b}^{2}+{c}^{2} \right)  \left(8\,{b}^{2}-{c}^{2}
 \right) {  q_{021}}\,{  q_{102}}\\
 -&8\,{b}^{2} \left(4\,{b}^{2}+{c}^{2}
 \right)  \left(2\,{b}^{2}-{c}^{2} \right) {  q_{111}}\,{  q_{102}}
 +12
\,{b}^{3}c \left(4\,{b}^{2}+{c}^{2} \right) {  q_{102}}\,{  q_{201}}\\
-&2
\,bc \left(4\,{b}^{2}+{c}^{2} \right)  \left(8\,{b}^{2}-{c}^{2}
 \right) {  q_{111}}\,{  q_{012}}
 - 48\,{b}^{4}c{{  q_{201}}}^{3}\\
 +&4\,{b}^{2
}c \left(-17\,{c}^{2}+28\,{b}^{2} \right) {{  q_{111}}}^{2}{  q_{201}}
+
c \left(44\,{b}^{2}{c}^{2}+32\,{b}^{4}-3\,{c}^{4} \right) {{  q_{111}}
}^{2}{  q_{021}}\\
+&16\,b \left(2\,b-c \right)  \left(2\,b+c \right) 
 \left({b}^{2}+{c}^{2} \right) {{  q_{111}}}^{3} 
 - 8\,b \left(8\,{b}^{4
}-12\,{b}^{2}{c}^{2}+{c}^{4} \right) {  q_{021}}\,{  q_{201}}\,{  q_{111}
}\\
+&96\,{b}^{3} \left(b-c \right)  \left(b+c \right) {{  q_{201}}}^{2}{
  q_{111}}+6\,b{c}^{4}{{  q_{021}}}^{2}{  q_{111}}
-4\,{b}^{2}c \left(8\,
{b}^{2}-{c}^{2} \right) {  q_{201}}\,{{  q_{021}}}^{2}\\
+&4\,{b}^{2}c
 \left(-{c}^{2}+20\,{b}^{2} \right) {{  q_{201}}}^{2}{  q_{021}}
 \endaligned
 $$
 }
 When $q_{111}=0$ and $q_{201}=q_{021}$ (necessary conditions for the integrability of the plane field ${\mathcal P}_3$)
 the coefficient $\chi_{12}^{*}$ is given by:
 
 $$\aligned\chi_{12}^{*}=& -4\,{b}^{2} \left(4\,{b}^{2}+{c}^{2} \right)  \left(2\,{b}^{2}-{c}^{
2} \right)  \left(k-{ k_3} \right) { Q_{211}}\\
-&bc \left(4\,{b}^{2}+
{c}^{2} \right)  \left(8\,{b}^{2}-{c}^{2} \right)  \left(k-{k_3}
 \right) Q_{121}\\
 +&4\,{b}^{3}c \left(4\,{b}^{2}+{c}^{2} \right) 
 \left(k- k_{3} \right)  Q_{301}-2\,{b}^{2}{c}^{2} \left(4\,{b}^
{2}+{c}^{2} \right)  \left(k-{ k_3} \right) Q_{031}\\
-&b q_{021}
\, \left(4\,{b}^{2}+{c}^{2} \right) ^{2} \left(2\, q_{012}\,b- q_{102}\,c \right) \endaligned$$

 Also the  coefficient $\chi_{12}$ is given by:
 $$\chi_{12}={b}^{2} \left(-{ k_3}\, Q_{211}+k Q_{211}+q_{012}\, q_{021
} \right) $$
 
\begin{remark} The long expressions appearing in this Appendix such as those in equations  
 \eqref{funcaoU}, \eqref{funcaoV}, and,        \eqref{chi11},  have been corroborated by Computer Algebra.
 
\end{remark}
\bibliographystyle{plain}

\vskip 1cm

\author{\noindent D\'ebora Lopes\\ Departamento de Matem\'atica\\ Universidade Federal do Sergipe\\
Av. Marechal Rondon, s/n Jardim Rosa Elze - CEP 49100-000\\
S\~ao Crist\'ov\~ao, SE, Brazil}
 \email{deb@deboralopes.mat.br}

\vskip 0.7cm
\author{\noindent Jorge Sotomayor\\ Instituto de Matem\'atica e Estat\'{\i}stica \\
Universidade  de S\~ao Paulo,\\
 Rua do Mat\~ao  1010,
Cidade Univerit\'aria, CEP 05508-090,\\
S\~ao Paulo, S. P, Brazil}
 \email{sotp@ime.usp.br}
 
 \vskip 0.7cm

\author{\noindent Ronaldo Garcia\\Instituto de Matem\'atica e Estat\'{\i}stica \\
Universidade Federal de Goi\'as,\\ CEP 74001--970, Caixa Postal 131 \\
Goi\^ania, Goi\'as, Brazil}
 \email{ragarcia@ufg.br}

\end{document}